\documentclass[11pt,a4paper]{amsart}  

\usepackage{amscd,amssymb,amsopn,amsmath,amsthm, graphics,amsfonts, enumerate,verbatim,calc}
\usepackage{graphicx}
\usepackage{url}
\usepackage{mathrsfs} 
\usepackage[all]{xy}

\usepackage{color}

\usepackage{mathtools} 
\usepackage{colonequals} 

\usepackage{tikz-cd}

\usepackage[utf8]{inputenc}
\usepackage[T1]{fontenc}
\usepackage{lmodern}
\usepackage[english]{babel}
\usepackage[autostyle]{csquotes}

\usepackage{xcolor}
\usepackage{hyperref}
\hypersetup{
  bookmarksnumbered=true, %
  pdftitle={}, 
  pdfsubject={}, 
  pdfauthor={}, 
  pdfkeywords={}, 
  colorlinks=true, 
  linkcolor=black,
  citecolor=black,
  urlcolor=black,
  bookmarksnumbered 
}
\usepackage{cleveref} 
\usepackage{mymacros_common}

\numberwithin{equation}{section} 

\newcommand{\citeSta}[1]{\cite[Tag \href{https://stacks.math.columbia.edu/tag/#1}{#1}]{Stacks}}

\newtheorem{theorem}{Theorem}[section]
\newtheorem{lemma}[theorem]{Lemma}
\newtheorem{proposition}[theorem]{Proposition}
\newtheorem{corollary}[theorem]{Corollary}
\newtheorem{maintheorem}{Main Theorem}

\newtheorem{question}{Question}
\newtheorem{Problem}{Problem}

\newtheorem{conjecture}[theorem]{Conjecture}

\theoremstyle{definition}
\newtheorem{definition}[theorem]{Definition}
\newtheorem{remark}[theorem]{Remark}
\theoremstyle{remark}

\newtheorem{example}[theorem]{Example}

\newcommand{\im}{\operatorname{Im}}
\newcommand{\Ker}{\operatorname{Ker}}
\newcommand{\Spec}{\operatorname{Spec}}

\newcommand{\id}{\operatorname{id}}
\newcommand{\Ext}{\operatorname{Ext}}

\newcommand{\Hom}{\operatorname{Hom}}

\newcommand{\coker}{\operatorname{coker}}

\newcommand{\Pic}{\operatorname{Pic}}

\newcommand{\Frac}{\operatorname{Frac}}

\newcommand{\ob}{\operatorname{ob}}
\newcommand{\Nr}{\operatorname{Norm}}

\newcommand{\Proj}{\operatorname{Proj}}

\newcommand{\Spf}{\operatorname{Spf}}
\newcommand{\Alb}{\operatorname{Alb}}

\newcommand{\prarrow}[2]{\ar@<0.5ex>[r]^-{#1}\ar@<-0.5ex>[r]_-{#2}} \newcommand{\plarrow}[2]{\ar@<0.5ex>[l]^-{#1}\ar@<-0.5ex>[l]_-{#2}} \newcommand{\pdarrow}[2]{\ar@<0.5ex>[d]^-{#1}\ar@<-0.5ex>[d]_-{#2}} \newcommand{\puarrow}[2]{\ar@<0.5ex>[u]^-{#1}\ar@<-0.5ex>[u]_-{#2}}
\newcommand{\coloneq}{\coloneq}

\newcommand{\perfd}{\mathrm{perfd}}

\begin{document}
\title[Quasi-canonical lifting of projective varieties in positive characteristic]
{Quasi-canonical lifting of projective varieties in positive characteristic}

\author[R. Ishizuka]{Ryo Ishizuka}
\address{Department of Mathematics, Institute of Science Tokyo, 2-12-1 Ookayama, Meguro, Tokyo 152-8551, Japan}
\email{ishizuka.r.ac@m.titech.ac.jp}

\author[K. Shimomoto]{Kazuma Shimomoto}
\address{Department of Mathematics, Institute of Science Tokyo, 2-12-1 Ookayama, Meguro, Tokyo 152-8551, Japan}
\email{shimomotokazuma@gmail.com}

\dedicatory{Dedicated to Prof. Kei-ichi Watanabe on the occasion of his 80th birthday}




\begin{abstract}
The main aim of this article is to give new classes of smooth projective varieties over characteristic \(p>0\) that admit flat liftings over the Witt vectors together with additional data (logarithmic structure and the Frobenius morphism) by showing a descending property of such Frobenius liftability. We establish a refined form of the classical result due to Mehta-Srinivas on the existence of canonical liftings. For this purpose, we also establish a result on the algebraization of certain $p$-adic formal schemes.
\end{abstract}

\maketitle

\setcounter{tocdepth}{3}
\tableofcontents

\section{Introduction}

The aim of this paper is to investigate the problem of finding new classes of smooth projective varieties over a perfect field of positive characteristic that admit a flat lifting over the Witt vectors. We also study the logarithmic version and possibility of lifting the Frobenius morphism over the Witt vectors. Some applications of the results in this article are given in \cite{IS25}, including constructions of Noetherian rings of mixed characteristic $p>0$ admitting a ring map which reduces to the $p$-th power map modulo $p$ (so-called \emph{p-torsion free $\delta$-ring} in the recent literature).

Let us start with a review on the deformation theory of abelian varieties of characteristic $p>0$ by Serre-Tate \cite{Ka81} and its conceptual generalization by Mehta-Nori-Srinivas \cite{MS87}. If $A$ is an abelian variety over a field $k$ of characteristic $p>0$ and $R$ is a local Artinian ring with residue field $k$, then the classical theory of Serre-Tate asserts that the deformation theory of $A$ along $\Spec(k) \to \Spec(R)$ can be read off from the corresponding data in the associated $p$-divisible group $A[p^\infty]$, which encodes a linear algebra structure of abelian varieties (see \cite[Theorem 1.2.1]{Ka81}). This theory can be strengthened if moreover $A$ is assumed to be \emph{ordinary} (see \Cref{ordinaryDef} for the general definition of ordinarity and see \cite[Theorem 2.1]{Ka81} and \cite[Corollary (1.2), p. 177]{Me72} for the results of ordinary abelian varieties). Recently, this theory has been extended over ordinary abelian schemes over any \(p\)-adic base (see \cite{BorGur19} and \cite{BorGur20}).

Following \cite{MS87}, we will use the following notion in this paper:

\begin{definition}[{\Cref{canolift}}] \label{Introcanolift}
  Let \(k\) be a perfect field of characteristic \(p > 0\) and let \(X\) be a scheme over \(k\).
  A \emph{flat lifting} of \(X\) is a flat \(W(k)\)-scheme \(\mcalX\) such that whose closed fiber is isomorphic to \(X\) as a \(k\)-scheme.
  A \emph{quasi-canonical lifting} of \(X\) over \(W(k)\) is a pair of a flat lifting \(\mcalX\) of \(X\) and a morphism \(\widetilde{F}_X:\mcalX \to \mcalX\) that lifts the absolute Frobenius morphism \(F_X:X \to X\) and compatible with the Witt-Frobenius on \(W(k)\).
  If \((\mcalX, \widetilde{F}_X)\) exists and is unique up to isomorphism, then we call it a \emph{canonical lifting} of \(X\) over \(W(k)\).
\end{definition}

In particular, the authors of \cite{MS87} consider the following situation. Let $X$ be a smooth projective variety over a perfect field $k$ of characteristic $p>0$ whose cotangent bundle is trivial (equivalently the tangent bundle, which is dual to the cotangent bundle, is trivial).
If \(X\) is ordinary (\Cref{ordinaryDef}), then there exists a \emph{canonical lifting} of \(X\) over \(W(k)\) (\cite[Appendix: Canonical liftings Theorem 1]{MS87}, see also \Cref{MSResultsTheorem}).

\begin{theorem}[Dupuy, Mehta-Srinivas]
\label{DMSTHm}
The following assertions hold.
\begin{enumerate}
\item
Let $\mathcal{X}$ be a smooth projective scheme flat over $W(\overline{\mathbb{F}}_p)$ such that the Kodaira dimension of $\mathcal{X}$ is positive. Then $\mathcal{X}$ does not admit a lift of the Frobenius morphism on the special fiber $\mathcal{X} \times_{W(\overline{\mathbb{F}}_p)} \overline{\mathbb{F}}_p$.

\item 
If $X$ is ordinary with a trivial cotangent bundle, then $X$ has a finite Galois covering $Y \to X$ of $p$-power degree such that $Y$ is an ordinary abelian variety.\footnote{In contrast, if the base field is of characteristic 0, then the Albanese mapping $X \to \Alb(X)$ is an isomorphism and thus, $X$ is an abelian variety.}
\end{enumerate}
\end{theorem}

The first statement is in \cite{Dup14} and the second one is in \cite{MS87}. Moreover, it is observed that the existence of a flat lifting over the Witt vectors is related to the \emph{globally Frobenius-split} condition:

\begin{definition}[Mehta-Ramanathan] \label{FrobSplit}
  A scheme \(X\) over a field \(k\) of characteristic \(p>0\) is called \emph{globally Frobenius-split} if the absolute Frobenius morphism \(F \colon \mathcal{O}_X \to F_*\mathcal{O}_X\) admits a splitting as a morphism of \(\mathcal{O}_X\)-modules.
\end{definition}

This is introduced by Mehta-Ramanathan \cite{MR85} in order to deduce the Kodaira vanishing theorem in positive characteristic and has been studied as one of the most fundamental global properties of varieties in positive characteristic.
We recall the following conjectures (see \cite[Conjecture (at \S 1) and Proposition 3.1.2]{AZ21} for example).

\begin{conjecture}[Achinger-Zdanowicz]
\label{FrobWittLift}
Let $X$ be a smooth projective variety defined over an algebraically closed field $k$ of characteristic $p>0$. Assume that $X$ is globally Frobenius-split with trivial canonical class. Then there is a flat deformation of $X$ over the Witt vectors $W(k)$ (not only \(W_2(k)\)).
\end{conjecture}

If \(X\) is dimension two or if \(X\) is a finite \'etale quotient of an abelian variety, then the conjecture is proved in \cite[Theorem 1.3 and Proposition 4.12]{BBKW22}. Another intriguing conjecture is the following (see \cite[Conjecture 1]{AWZ21}).

\begin{conjecture}[Achinger-Witaszek-Zdanowicz]
\label{FrobWittLift2}
Let $X$ be a smooth projective variety defined over an algebraically closed field $k$ of characteristic $p>0$ with the Frobenius morphism $F_X:X \to X$. Assume that the pair $(X,F_X)$ admits a flat lifting over $W_2(k)$. Then there exists a finite \'etale Galois cover $f:Y \to X$ such that the Albanese morphism $Y \to \Alb(Y)$ admits a structure of a toric fibration. If $X$ is simply connected, then $X$ is a toric variety.
\end{conjecture}

We note that a complete classification of smooth projective surfaces over \(k\) which admit a flat lifting over \(W_2(k)\) and a lifting of the Frobenius morphism is proved in \cite[Theorem 6.9]{AWZ23}. See also \cite{Xin16}. Although we are interested in the lifting problem over $W(k)$, we hope that our methods sheds light on \Cref{FrobWittLift2}. Our guiding principle is the following problem.

\begin{Problem}
\label{DescentProblem}
Let $X \to Y$ be a finite \'etale surjection of varieties over a perfect field $k$ of characteristic $p>0$ with the ring of Witt vectors $W(k)$. Assume that $X$ has a flat lifting $\mathcal{X}/W(k)$ with a morphism $\widetilde{F}_X:\mathcal{X} \to \mathcal{X}$ that lifts the absolute Frobenius on $X$. Then Is it true that there exists a flat lifting \(\mcalY\) of \(Y\) to \(W(k)\) with a lift of the Frobenius \(\widetilde{F} \colon \mcalY \to \mcalY\)?
\end{Problem}

Note that an answer to the converse direction of \Cref{DescentProblem} is given in \Cref{liftcanonical}. Namely, if the quotient variety \(Y\) admits a projective flat covering \(\mathcal{Y}\) over \(W(k)\) with a Frobenius lift \(\widetilde{F}_Y\), then so does \(X\) with a finite \'etale surjection \(\mathcal{X} \to \mathcal{Y}\) compatible with the Frobenius lifts.

\subsection{Descending property of liftings}

In order to study schemes over the Witt vectors, we consider two main themes. The first one is the existence of $p$-adic formal schemes. The second one is the algebraization problem.

First, we prove a fundamental result on algebraizations over the Witt vectors for $p$-adic formal schemes arising from (singular) proper varieties over an algebraically closed field of characteristic $p>0$ without assuming cohomological data.
To explain the result, we introduce the following definition of \(p\)-adic formal schemes. Our formulation follows \cite[\S 2.2.3]{Dri22} and \cite[Definition 24.82]{GW2}:

\begin{definition} \label{DefFormalSchemes}
  Let \(k\) be a perfect field of characteristic \(p > 0\) and set \(S_n \defeq \Spec(W_n(k))\).
  A \emph{\(p\)-adic formal scheme \(\{X_n\}_{n \geq 1}\) over \(W(k)\)} is a sequence of schemes \(X_1 \to X_2 \to \cdots\) such that each \(X_n\) is a flat \(S_n\)-scheme and \(X_n \cong X_{n+1} \times_{S_{n+1}} S_n\) for all \(n \geq 1\).
  A morphism of \(p\)-adic formal schemes \(\{f_n\} \colon \{X_n\}_{n \geq 1} \to \{Y_n\}_{n \geq 1}\) is a sequence of morphisms \(f_n \colon X_n \to Y_n\) over \(S_n\) compatible with the transition morphisms \(X_n \to X_{n+1}\) and \(Y_n \to Y_{n+1}\) for all \(n \geq 1\).

  A scheme over \(\mcalX\) over \(S \defeq \Spec(W(k))\) is called an \emph{algebraization} of a \(p\)-adic formal scheme \(\{X_n\}_{n \geq 1}\) if there is an isomorphism of \(p\)-adic formal schemes \(\{X_n\}_{n \geq 1} \cong \{\mcalX \times_S S_n\}_{n \geq 1}\).
\end{definition}

Our first main theorem is as follows.

\begin{maintheorem}[Algebraization of $p$-adic formal schemes]
\label{algquotient}
  Let \(k\) be a perfect field of characteristic \(p > 0\) and let \(f:X \to Y\) be a surjective finite \'etale morphism of integral proper varieties over \(k\) such that the degree $d \defeq [K(X):K(Y)]$ is not divisible by \(p\).
Then the following assertions hold.
\begin{enumerate}
\item
Assume the following condition:
\begin{enumerate}
\item[$\bullet$] If $X = X_1 \to X_2 \to \cdots \to X_k$ is any sequence of schemes such that $X_j$ is a flat $S_j$-scheme and $X_j \cong X_{j+1} \times_{S_{j+1}} S_j$ for any $1 \le j \le k-1$, then there is a morphism $X_k \to X_{k+1}$ such that $X_{k+1}$ is a flat $S_{k+1}$-scheme and $X_{k} \cong X_{k+1} \times_{S_{k+1}} S_k$.
\end{enumerate}
Then there is a morphism of $p$-adic formal schemes \(\{f_n\}_{n \geq 1} \colon \{Z_n\}_{n \geq 1} \to \{Y_n\}_{n \geq 1}\) such that \(f_1 \colon Z_1 \to Y_1\) is isomorphic to \(f \colon X \to Y\)
and the following properties are satisfied.
\begin{enumerate}
\item Each square 
\[ 
\xymatrix@M=10pt{ 
Z_j \ar[r]\ar[d]_{f_j}&Z_{j+1} \ar[d]_{f_{j+1}}&\\
Y_j \ar[r]&Y_{j+1} &\\
}
\]
is a pullback diagram in the category of schemes and each \(f_j \colon Z_j \to Y_j\) is uniquely determined by the data \(f \colon X \to Y\) for all \(j > 1\).

\item
Each morphism $f_j \colon Z_j \to Y_j$ is a surjective finite \'etale morphism of constant degree $d$ for all $j \geq 1$.\footnote{Any finite \'etale map \(f:X \to Y\) of schemes can be written as \(\bigsqcup_{j=1}^N Y \to Y\) \'etale locally on the target by \citeSta{04HN}. We say that \(f\) has \emph{constant degree \(d\)} if the number \(N\) is constantly \(d\).}
\end{enumerate}
\item
Suppose that \(\{Z_n\}_{n \geq 1}\) above is algebraizable by a flat projective \(W(k)\)-scheme $\mathcal{Z}$. Then \(\{Y_n\}_{n \geq 1}\) admits an algebraization $\mathcal{Y}$ which is a projective flat $W(k)$-scheme. Furthermore, there exists a finite \'etale surjective morphism $\widetilde{f} \colon \mathcal{Z} \to \mathcal{Y}$ over \(W(k)\) whose reduction along $\Spec(k) \to \Spec(W(k))$ is $f \colon X \to Y$.
\end{enumerate}
\end{maintheorem}

The next result gives a partial answer to Conjecture \ref{FrobWittLift} (together with Frobenius lifts) in the logarithmic setting, which also gives a substantial variation of the results of Mehta-Nori-Srinivas on the existence of the canonical lifting of ordinary projective varieties with trivial cotangent bundle \cite{MS87}.

Let $(X,D,F_X)$ be a triple, where $X$ is a smooth proper variety over a perfect field $k$ of characteristic $p>0$ together with a normal crossing divisor $D$, and $F_X$ is the Frobenius morphism on $X$.
Similarly as in \Cref{Introcanolift}, we introduce the following definition of a quasi-canonical lifting of $(X,D,F_X)$ over $W(k)$:

\begin{definition}[\Cref{Defloglift}] \label{DefQuasiCanonicalLiftingDivisor}
  A \emph{quasi-canonical lifting} of \((X, D, F_X)\) over \(W(k)\) is a triple $(\mathcal{X},\mathcal{D},\widetilde{F}_X)$, where \((\mcalX, \widetilde{F}_X)\) is a quasi-canonical lifting of \(X\) and $\mathcal{D} \subseteq \mathcal{X}$ is a divisor with normal crossings relative to $\Spec(W(k))$ such that $D=\mathcal{D}|_X$ along the closed immersion $X \hookrightarrow \mathcal{X}$ and $\widetilde{F}^*_X\mathcal{D}=p\mathcal{D}$.
  If it exists and is unique up to isomorphism, then we call it a \emph{canonical lifting} of \((X, D, F_X)\) over \(W(k)\).
\end{definition}

\begin{example}[Non-triviality of the condition $\widetilde{F}_{X}^* \mathcal{D} = p \mathcal{D}$ on $\mathbb{P}^1$]
To see that the condition $\widetilde{F}_{X}^* \mathcal{D} = p \mathcal{D}$ is highly non-trivial, consider $X = \mathbb{P}^1_k$ over a perfect field $k$ of characteristic $p > 2$. Let $D = \{1\} \subset \mathbb{P}^1_k$ be the divisor consisting of a single point, which lies in the standard affine open patch $U_0 = \mathbb{A}^1_k = \operatorname{Spec} k[x]$. We choose a quasi-canonical lifting of $(X, F_X)$ as $(\mathcal{X}, \widetilde{F}_{X}) = (\mathbb{P}^1_{W(k)}, \widetilde{F}_{X})$, whose restriction to the affine patch $\mathcal{U}_0 = \mathbb{A}^1_{W(k)} = \operatorname{Spec} W(k)[x]$ is equipped with the standard Frobenius lift $(a:b) \mapsto (a^p:b^p)$.

Now, let $\mathcal{D} \subset \mathcal{X}$ be a relative normal crossings divisor lifting $D$. Restricted to $\mathcal{U}_0$, such a divisor must be defined by the ideal $\mathcal{I}(\mathcal{D}) = (x - a)$ for some element $a \in W(k)$ satisfying $a \equiv 1 \pmod p$. On this affine patch, the condition $\widetilde{F}_{X}^* \mathcal{D} = p \mathcal{D}$ means that the local defining functions must satisfy
$$
\widetilde{F}_{X}^{*}(x-a)=u\cdot (x-a)^{p}
$$
for some invertible element $u \in W(k)[x]^\times = W(k)^\times$. Expanding both sides, we obtain:
$$
x^{p}-\widetilde{F}_{X}^{*}(a)=u\cdot \Big(x^{p}-pax^{p-1} + \binom{p}{2}a^2 x^{p-2} - \dots - a^{p}\Big).
$$
By comparing the coefficients of $x^{p}$, we must have $u = 1$. This forces all the remaining terms to match exactly. Looking at the coefficient of $x^{p-1}$, we obtain the relation:
$$
pa = 0 \quad \text{in } W(k).
$$
However, since $W(k)$ is $p$-torsion-free and $a \equiv 1 \pmod p$, $pa$ is never zero in $W(k)$. 
\end{example}


Our second main theorem (\Cref{mtheorem1}) gives a sufficient condition for the existence of a quasi-canonical lifting of a smooth projective variety with a normal crossing divisor.
This theorem uses the notion of normal crossing pairs (nc pairs) which is \Cref{Defncpair} of the text and the following notion of coverings of nc pairs:

\begin{definition} \label{DefGoodFrobeniusLiftingCover}
  Let \((X, D)\) be a smooth projective nc pair over an algebraically closed field \(k\) of characteristic \(p > 0\).
  A \emph{good Frobenius lifting cover} is a surjective finite \'etale morphism \(f \colon Z \to X\) of smooth projective \(k\)-varieties satisfying the following conditions:
  \begin{enumerate}
    \item The nc pair \((Z, D_Z  \defeq  f^*D)\) admits a quasi-canonical lifting \((\mathcal{Z}, \mathcal{D}_Z, \widetilde{F}_Z)\) over \(W(k)\).
    \item Let \(T_Z(-\log D_Z)\) be the dual to the logarithmic cotangent bundle \(\Omega^1_Z(\log D_Z)\) with logarithmic poles on \(D_Z\) and let \(B\Omega^1_Z\) be the image of the differential \(F_*d \colon F_*\mcalO_Z \to F_*\Omega^1_Z\) (\Cref{CartierIsom}). Then the vanishing holds: \(H^0(Z, T_Z(-\log D_Z) \otimes B\Omega^1_Z) = H^1(Z, T_Z(-\log D_Z) \otimes B\Omega^1_Z) = 0\).
    \item The degree \([K(Z) : K(X)]\) is not divisible by \(p\).
  \end{enumerate}
\end{definition}

By \Cref{CanonicalUnique}, the quasi-canonical lifting of \((Z, D_Z)\) is uniquely determined and thus, it is the canonical lifting.

\begin{maintheorem}
\label{mtheorem1}
Let $(X,D)$ be a smooth projective nc pair over an algebraically closed field $k$ of characteristic $p>0$ and let \(f \colon Z \to X\) be a finite \'etale surjective morphism.
If there exists a good Frobenius lifting cover \(f \colon Z \to X\) of \((X, D)\) (\Cref{DefGoodFrobeniusLiftingCover}), then there exists a canonical lifting \((\mathcal{X}, \mathcal{D}, \widetilde{F}_X)\) of \((X, D, F_X)\) over \(W(k)\) and a finite \'etale surjective morphism \(\widetilde{f} \colon \mathcal{Z} \to \mathcal{X}\) compatible with the canonical liftings \((\mcalZ, \mcalD_Z, \widetilde{F}_Z)\) and \((\mcalX, \mcalD, \widetilde{F}_X)\).

\end{maintheorem}

\begin{remark}
The introduction of the logarithmic structure is indispensable for expanding the scope of Main Theorem 2 beyond ordinary abelian varieties. As detailed in \cite[Section 1.3 and Theorem 5.1.1]{AWZ21}, if $D=0$, a smooth projective variety with a trivial tangent bundle is essentially restricted to an ordinary abelian variety up to finite étale covers by Winkelmann's theorem and Mehta-Srinivas theorem (see Theorem \ref{DMSTHm}).

On the other hand, for a smooth projective toric variety $Z$ (e.g., $Z = \mathbb{P}^1 \times \mathbb{P}^1$), the ordinary tangent bundle $T_Z$ is never trivial, whereas the logarithmic tangent bundle $T_Z(-\log D_Z)$ with respect to the full toric boundary $D_Z$ is completely trivial. For example, let \(E\) be an ordinary elliptic curve, set \(Z=E\times\mathbb{P}^1_k\) and \(D_Z=E\times\{0,\infty\}\), and let a non-trivial finite subgroup \(G\subset E(k)\) of order prime to \(p\) act by translations on the first factor. The quotient is \((X,D)=((E/G)\times\mathbb{P}^1_k,(E/G)\times\{0,\infty\})\). The product of the canonical lifting of \(E\) and the standard Frobenius lifting of \((\mathbb{P}^1_k,\{0,\infty\})\) lifts \((Z,D_Z)\). Moreover, \(T_Z(-\log D_Z)\cong\mathcal{O}_Z^{\oplus2}\), and \(H^i(Z,B\Omega_Z^1)=0\) for \(i=0,1\). Thus \(Z\to X\) is a good Frobenius lifting cover, whereas \(X\) is not toric since \(H^1(X,\mathcal{O}_X)\ne0\).

Moreover, if \(X\) in the statement of \Cref{mtheorem1} has a finite \'etale cover \(Z\) which is an ordinary abelian variety, then \(X\) has a quasi-canonical lifting and a functoriality of canonical liftings in \Cref{AbVarLiftWitt} in which the degree condition (\Cref{DefGoodFrobeniusLiftingCover} (3)) is not required. 
This is claimed in \cite[Example 3.1.4 and Remark 3.1.6]{AWZ21} without proof.
\end{remark}

The proof of \Cref{mtheorem1} relies on descent of quasi-canonicity along \'etale morphisms and the deformation theory of formal schemes via cotangent complexes as well as \Cref{algquotient}. Along the way, we prove that the quasi-canonical property ascends along a finite \'etale morphism in a compatible manner (see \Cref{liftcanonical}), which is of independent interest. The main results in this paper will be applied in the construction of singularities in mixed characteristic $p>0$ in \cite{I23}.

\subsection*{Acknowledgement}
The authors would like to thank Shou Yoshikawa for providing useful comments. The authors also thank the referee for making a detailed list of corrections.
The first-named author was supported by JSPS KAKENHI Grant number 24KJ1085. The second-named author was supported by
JSPS KAKENHI Grant Number 26K06738.

\subsection*{AI disclosure}
The original mathematical questions, statements, and proofs, except for \Cref{ConverseEtaleQuotient}, were developed by the authors without AI assistance. During revision, ChatGPT was used to assist with checking the manuscript and assist in formulating and proving \Cref{ConverseEtaleQuotient}. The authors take full responsibility for the final manuscript.

\section{A review of lifting of Frobenius morphisms}

\subsection{Quasi-canonical lifting}
We give a review of the theory of canonical liftings of projective varieties with its Frobenius morphism. Let $X$ be an $\mathbb{F}_p$-scheme and let $F_X:X \to X$ denote the absolute Frobenius morphism. If $f:X \to Y$ is a morphism of $\mathbb{F}_p$-schemes, then there is a commutative diagram of $\mathbb{F}_p$-schemes:
\begin{center}
\begin{tikzcd}
X
\arrow[drr, bend left, "F_X"]
\arrow[ddr, bend right, "f"]
\arrow[dr, dotted,  "F_{X/Y}"description] & & \\
& X^{(1)} \arrow[r, "F_Y^{(1)}"] \arrow[d]
& X \arrow[d, "f"] \\
& Y \arrow[r, "F_Y"]
& Y
\end{tikzcd}
\end{center}
where the square is cartesian. We say that $F_{X/Y}$ is the \emph{relative Frobenius morphism} of $X/Y$. In what follows, we also write $F$ for $F_X$. Let $k$ be a perfect field of characteristic $p>0$.

Recall that a variety $X$ over $k$ is \emph{globally Frobenius-split} if the natural $\mathcal{O}_X$-linear map $\mathcal{O}_X \to F_*\mathcal{O}_X$ splits (\Cref{FrobSplit}).

Assume that $X$ is a smooth variety over $k$ and let $\Omega^1_X$ be a locally free $\mathcal{O}_X$-module consisting of K\"ahler differential 1-forms. We also write $\Omega^1_X$ for $\Omega^1_{X/k}$ if no confusion is likely to occur. We have the de Rham complex which consists of sheaves of differential forms $\{\Omega^i_X\}_{i \ge 0}$ and the differential maps $d^i:\Omega^i_X \to \Omega^{i+1}_X$. By pushing forward along the Frobenius $F:X \to X$, we get a system of coherent $\mathcal{O}_X$-modules $\{F_*\Omega^i_X\}_{i \ge 0}$.
We recall the following ``Cartier isomorphism''. Remark that its logarithmic version also holds (see \cite[Theorem (4.12)]{Kato96}).

\begin{theorem}[{\cite[Theorem (7.2)]{Ka70}}] \label{CartierIsom}
  In the setting above, let
  \begin{equation*}
    B\Omega^i_X\coloneqq \im \big(F_*d^{i-1}:F_*\Omega^{i-1}_X \to F_*\Omega^{i}_X\big)~\mbox{and}~Z\Omega^i_X\coloneqq \Ker \big(F_*d^{i}:F_*\Omega^{i}_X \to F_*\Omega^{i+1}_X\big).
  \end{equation*}
  Then there is an exact sequence of locally free $\mathcal{O}_X$-modules
  \begin{equation}
  \label{CartierExact1}
    0 \to B\Omega^i_X \to Z\Omega^i_X \xrightarrow{C} \Omega^i_X \to 0,
  \end{equation}
  where $C$ is called the Cartier operator. In particular, this induces an isomorphism known as ``Cartier isomorphism'' $\bigoplus_{i \ge 0}\mathcal{H}^i((F_{X/k})_*\Omega_{X/k}^\bullet) \cong \bigoplus_{i \ge 0} \Omega^i_{X^{(p)}/k}$.
\end{theorem}

There is another exact sequence
\begin{equation}
\label{CartierExact3}
0 \to Z\Omega^i_X \to F_*\Omega^i_X \to B\Omega^{i+1}_X \to 0
\end{equation}
by the definition of $Z\Omega^i_X$ and $B\Omega^{i+1}_X$. In particular, since \(X\) is smooth and then \(Z\Omega^0_X = \Ker(F_*\mcalO_X \xrightarrow{F_*d^0} F_*\Omega_X^1)\) is isomorphic to \(\mcalO_X\) along the Frobenius morphism, we have the exact sequence
\begin{equation*} \label{CartierExact2}
  0 \to \mathcal{O}_X \to F_* \mathcal{O}_X \to B\Omega^1_X \to 0.
\end{equation*}
Let $\omega_X\coloneqq \bigwedge^n \Omega^1_{X}$ be the canonical sheaf of $X$. This is an invertible sheaf. We recall the ordinarity condition after Bloch-Kato \cite[Definition (7.2)]{BK86} and Illusie-Raynaud \cite[Definition (4.12)]{IR83} following \cite[Definition 8.8]{BT24}.

\begin{definition}
\label{ordinaryDef}
Let \(n \geq 0\). We say that a smooth projective variety over $k$ is \emph{\(n\)-ordinary} if $H^i(X,B\Omega^j_X)=0$ for all $i \ge 0$ and $j \leq n$. If $X$ is \(n\)-ordinary for all \(n \geq 0\), then we say that \(X\) is \emph{ordinary}, namely, \(H^i(X, B\Omega_X^j) = 0\) for all \(i \geq 0\) and \(j \geq 0\).
\end{definition}

According to \cite[Th\'eor\`em (4.13)]{IR83} and \cite[Proposition (7.3)]{BK86}, if the crystalline cohomology \(H^q_{\crys}(X/W(k))\) is torsion-free for all \(q\), this ordinarity condition is equivalent to the coincidence of the Newton and Hodge polygons of $X$. 
For an abelian variety, this is equivalent to the classical ordinarity condition requiring that \(X[p](\overline{k}) = (\setZ/p\setZ)^{\dim X}\) (\cite[Example (7.4)]{BK86}).

The next lemma recalls \cite[Lemma 1.1]{MS87}, together with its formulation in terms of Frobenius liftings over $W_2(k)$.

\begin{lemma}[Mehta-Srinivas]
\label{ordinarysplit}
Suppose $X$ is a smooth projective variety over an algebraically closed field $k$ of characteristic $p>0$ such that $\Omega_X^1$ is trivial. Then the following conditions are equivalent. 
\begin{enumerate}
\item
$X$ is globally Frobenius-split.

\item
$X$ is ordinary.

\item
The Frobenius action on $H^{\dim X}(X,\mathcal{O}_X)$ is bijective.

\item
The exact sequence $0 \to B\Omega^1_X \to Z\Omega^1_X \xrightarrow{C} \Omega^1_X \to 0$ splits as $\mathcal{O}_X$-modules.

\item
\((X, F)\) has a lifting \((X_2, F_2)\) over \(W_2(k)\) (in the sense of \Cref{canolift}).
\end{enumerate}
\end{lemma}

\begin{proof}
The equivalence of (1), (2), and (4) follows from \cite[Lemma 1.1]{MS87}. Since $\omega_X$ is trivial, (1) is equivalent to (3) by \cite[Proposition 3.1.2]{AZ21}. Finally, (4) is equivalent to (5) by \cite[Proposition 3.3.1(c)]{AWZ21}. 
\end{proof}

The notion of (quasi-)canonical liftings will play a central role. There seems to be several different versions in the literature of (quasi-)canonical liftings depending on the purpose. For example, \(F\)-liftings and canonical liftings in the classical Serre--Tate theory.
Since our primitive purpose is whether the Frobenius morphism on a smooth projective variety over \(k\) can be lifted over \(W(k)\) and one of the fundamental results in this direction is the work of Nori--Srinivas in \cite[Appendix: Canonical liftings]{MS87} (see also \Cref{MSResultsPropositions} and \Cref{MSResultsTheorem} below), we adopt the following definition of (quasi-)canonical liftings which is based on their work.
See \Cref{RemarkLiftings} for a discussion on the relation to other versions of liftings in  the literature.

\begin{definition}
\label{canolift}
Let $X$ be a scheme over a perfect field $k$ of characteristic $p>0$. 
\begin{enumerate}
\item
A scheme $\mathcal{X}$ is said to be a \emph{flat lifting} of $X$ over $W(k)$ if there is a flat surjective morphism $f:\mathcal{X} \to \Spec(W(k))$ such that the closed fiber of $f$ is isomorphic to $X$ as a \(k\)-scheme. If \(\mathcal{X}\) is projective over \(W(k)\), we say that \(\mathcal{X}\) is a \emph{projective flat lifting} of \(X\).

\item
A pair $(\mathcal{X},\widetilde{F}_X)$ is a \emph{quasi-canonical lifting} of $X$ if $\mathcal{X}$ is a flat lifting of $X$ and $\widetilde{F}_X:\mathcal{X} \to \mathcal{X}$ is a lifting of the absolute Frobenius $F_X: X \to X$ for which the diagram
\[
\xymatrix@M=10pt{
\mathcal{X} \ar[r]^{\widetilde{F}_X} \ar[d] &\mathcal{X} \ar[d]\\
\Spec(W(k)) \ar[r]^{\widetilde{F}_k} & \Spec(W(k)) \\
}
\]
commutes, where $\widetilde{F}_k$ is the unique lifting of the Frobenius map on $\Spec(k)$. If $\mathcal{X}$ is a projective scheme over $\Spec(W(k))$, then we say that $(\mathcal{X},\widetilde{F}_X)$ is a  \emph{projective quasi-canonical lifting}. 

\item If $(\mathcal{X},\widetilde{F}_X)$ exists uniquely up to isomorphism, then we say that it is a \emph{canonical lifting}. We call $\widetilde{F}_X$ a \emph{Frobenius lifting} of $F_X$. 

\item A morphism of quasi-canonical liftings $\widetilde{f}:(\mathcal{X},\widetilde{F}_X) \to (\mathcal{Y},\widetilde{F}_Y)$ is a morphism $\widetilde{f}:\mathcal{X} \to \mathcal{Y}$ of schemes over \(W(k)\) compatible with \(\widetilde{F}_X\) and \(\widetilde{F}_Y\).
\end{enumerate}

Also, we extend these definitions mutatis mutandis to flat lifting of $X$ over $W_n(k)$ with $n \in \mathbb{N}$.
Moreover, in the case of proper schemes, we assume the following condition in addition to the above definitions:
\begin{enumerate}
\item[$\bullet$]
When we consider the case that $X$ is proper, then we assume that a quasi-canonical lifting $f:\mathcal{X} \to \Spec(W(k))$ is a \textit{proper} morphism. 
\end{enumerate}
\end{definition}

Instead of working with the absolute Frobenius morphism $F_X:X \to X$, one can define a lifting of the relative Frobenius morphism $F_{X/k}:X \to X^{(1)}$ over $W(k)$ (or over $W_n(k)$). Indeed, the universality of the diagram defining $F_{X/k}$ implies that these two liftings are essentially equivalent. Notice that $F_{X/k}$ is a morphism of $k$-schemes, while $F_X$ is not.

\begin{remark} \label{RemarkLiftings}
  We clarify the distinctions between various versions of liftings in the literature.
  
  In \cite[Appendix: Canonical liftings]{MS87}, they use the term \emph{lifting} over \(W_n(k)\) (or over \(W(k)\)) to refer to what we call a flat lifting.
  Also, they use the term \emph{canonical lifting} to refer to what we call a canonical lifting in the sense of Definition \ref{canolift}.
  According to this terminology, a quasi-canonical lifting in our sense is a flat lifting over \(W(k)\) that admits a Frobenius lifting, but it may \emph{not} be unique.

  On the other hand, there is a notion of \emph{\(F\)-liftability} of a scheme \(X\) over \(k\) which requires the existence of a lifting \(\mathcal{X}\) of \(X\) over \(W_2(k)\) together with a lifting of the Frobenius morphism \(F_X\) to \(\mathcal{X}\). This terminology is used in, for example, \cite{Z17,AWZ21,AWZ23,BBKW22}.

  In the classical Serre--Tate theory, any ordinary abelian variety \(A\) over \(k\) admits a projective abelian scheme \(\mcalA\) on \(W(k)\) with the bijection \(\End_{W(k)}(\mcalA) \xrightarrow{\simeq} \End_k(A)\), which is constructed explicitly and unique up to canonical isomorphism. This is given in \cite[Theorem (3.3) and Corollary (3.4)]{Me72} and called the \emph{canonical lifting} of \(A\). See also \cite[Appendix]{Me72} for characterizations of the canonical lifting of an ordinary abelian variety. In particular, the Frobenius morphism on \(A\) uniquely lifts to \(\mcalA\) and thus \(\mcalA\) is a canonical lifting in the sense of \Cref{canolift}.

  By \cite[Theorem 2.1]{Ka81}, a lifting of an ordinary abelian variety corresponds to a Serre--Tate parameter \(q\), where the canonical lifting precisely corresponds to \(q=1\). In \cite[\S 3.2]{dN91}, de Jong and Noot define \emph{quasi-canonical lifting} of an ordinary abelian variety \(A\) over \(k\) to refer to a flat lifting \(\mcalA\) over \(W(k)\) whose Serre--Tate parameter has finite order (i.e., \(q^M=1\) for some integer \(M \ge 1\)). This algebraic condition is equivalent to the existence of a lifting of a \emph{power} of the Frobenius morphism on \(A\) to \(\mcalA\). This is clearly different from the notion of quasi-canonical lifting in \Cref{canolift}, where we strictly require the lifting of the absolute Frobenius morphism itself. Since we do not consider liftings of powers of the Frobenius morphism, we think there is no risk of confusion in this paper.

\end{remark}

In the rest of this paper, we use the results on Mehta--Srinivas's work in \cite[Appendix: Canonical liftings]{MS87} which are recalled here for the convenience of the reader.

\begin{proposition}[{\cite[Appendix: Canonical liftings, Proposition 1, 2, 3]{MS87}}] 
\label{MSResultsPropositions}
Let $X$ be an algebraic variety over a perfect field $k$ of characteristic $p>0$. Then the following assertions hold.
\begin{enumerate}
\item \textbf{Obstruction of a quasi-canonical lifting.}
Assume that $X$ is smooth. Let $(Y,F_{Y})$ be a quasi-canonical lifting of $X/k$ over $W_{n-1}(k)$. Then the obstruction to the existence of a quasi-canonical lifting $(Z,F_Z)$ of $X/k$ over $W_n(k)$ such that the restriction of $(Z,F_Z)$ along $\Spec(W_{n-1}(k)) \hookrightarrow \Spec (W_{n}(k))$ is $(Y,F_Y)$ is given by a cohomology class in $H^1(X,T_X \otimes B\Omega_X^1)$. Moreover, the set of all liftings as above forms a principal homogeneous space under $H^0(X,T_X \otimes B\Omega_X^1)$.

\item \textbf{Obstruction of lifting a line bundle.}
Let $(Z,F_Z)$ be a quasi-canonical lifting of $X/k$ over $W_n(k)$. Set the pair $(Y,F_Y)$ to be the restriction of $(Z,F_Z)$ along $\Spec(W_{n-1}(k)) \hookrightarrow \Spec (W_{n}(k))$. Assume that the Frobenius action on $H^1(X,\mathcal{O}_X)$ and $H^2(X,\mathcal{O}_X)$ are bijective. Let $\mathcal{L}$ be a line bundle on $Y$ such that $F_Y^*\mathcal{L}=\mathcal{L}^{\otimes p}$. Then there exists a line bundle $\mathcal{M}$ on $Z$ such that $F_Z^*\mathcal{M}=\mathcal{M}^{\otimes p}$ and $\mathcal{M}$ restricts to $\mathcal{L}$ along $\Spec(W_{n-1}(k)) \hookrightarrow \Spec(W_n(k))$.

\item \textbf{Obstruction of lifting a Frobenius-compatible morphism}
Suppose we are given smooth varieties $X_1/k$ and $X_2/k$. Let $(Z_1,F_{Z_1})$ and $(Z_2,F_{Z_2})$ be quasi-canonical liftings of $X_1/k$ and $X_2/k$ over $W_n(k)$, respectively. Let $Y_1, Y_2, F_{Y_1}$ and $F_{Y_2}$ be restrictions of $Z_1, Z_2, F_{Z_1}$ and $F_{Z_2}$ along $\Spec(W_{n-1}(k)) \hookrightarrow \Spec (W_{n}(k))$, respectively. Let $\phi:X_1 \to X_2$ be a $k$-morphism. Suppose we are given a $W_{n-1}(k)$-morphism $\psi:Y_1 \to Y_2$ lifting $\phi$ over $W_{n-1}(k)$, and $F_{Y_2} \circ \psi=\psi \circ F_{Y_1}$. Then the obstruction to the existence of $\chi:Z_1 \to Z_2$ lifting $\psi$ as a $W_n(k)$-morphism such that $F_{Z_2} \circ \chi=\chi \circ F_{Z_1}$ is given by a cohomology class in $H^0(X_1,\phi^*T_{X_2} \otimes B\Omega_{X_1}^1)$.

  \end{enumerate}
\end{proposition}

\begin{theorem}[{\cite[Appendix: Canonical liftings, Theorem 1]{MS87}}] 
\label{MSResultsTheorem}
Let $X$ be a smooth projective variety over a perfect field $k$ of characteristic $p>0$ such that $\Omega_X^1$ is a trivial bundle and $X$ is ordinary. Then the following assertions hold.
\begin{enumerate}
    \item 
    There exists a canonical lifting $(\mathcal{X},\widetilde{F}_X)$ of $X$ such that $\mathcal{X}$ is smooth projective over $W(k)$.

    \item 
    Let $X_1$ and $X_2$ be smooth projective varieties over $k$ such that $X_i$ has trivial $\Omega_{X_i}^1$ and ordinary for $i=1,2$. Let $\phi:X_1 \to X_2$ be a  $k$-morphism. Then there exists a unique $W(k)$-morphism $\psi:\mathcal{X}_1 \to \mathcal{X}_2$ over canonical liftings such that $\widetilde{F}_{X_2} \circ \psi=\psi \circ \widetilde{F}_{X_1}$ and the restriction of $\psi$ along $\Spec(k) \to \Spec(W(k))$ is $\phi$.

    \item 
    Let $(\mathcal{X},\widetilde{F}_{X})$ be a canonical lifting of $X$. Set $P:=\{\mathcal{L} \in \Pic(\mathcal{X})~|~\widetilde{F}_X^*(\mathcal{L}) \cong \mathcal{L}^{\otimes p}\}$. Then the composite mapping $P \hookrightarrow \Pic(\mathcal{X}) \to \Pic(X)$ is an isomorphism.
\end{enumerate}

\end{theorem}

We record the following lemma for convenience.

\begin{lemma}[{\cite[Lemma 24.96]{GW2}}]
\label{OpenNeighborhoodofClosedFiber}
Let \(f \colon \mathcal{X} \to \Spec(V)\) be a closed morphism of schemes, where \((V, sV, k)\) is a rank-$1$ valuation ring. Take the closed fiber \(f_{s} \colon \mathcal{X}_{s} \to \Spec(k)\). Then there is no open neighborhood of \(\mcalX_s\) in \(\mathcal{X}\) other than \(\mathcal{X}\).
\end{lemma}

\subsection{Logarithmic differentials}

We follow \cite[Definition 2.3.1]{AWZ21} for the following definition.

\begin{definition} \label{Defncpair}
Let \(S\) be a scheme and let $X$ be a smooth $S$-scheme.
A divisor \(D \subseteq X\) is said to be a \emph{divisor with normal crossings relative to \(S\)} if \'etale-locally on \(X\) there is an \'etale morphism \(h:X \to \mathbb{A}_S^n\) such that \(D=h^*(\{x_1 \cdots x_n=0\})\), where \(x_1,\ldots,x_n\) is the standard coordinate functions on \(\mathbb{A}_S^n\).
An \textit{nc pair} over $S$ is a pair $(X,D)$ such that $D \subseteq X$ is a divisor with normal crossings relative to $S$.
\end{definition}


We have a logarithmic variant of Definition \ref{canolift}. For instance, we can make the following.

\begin{definition} \label{Defloglift}
Let $(X,D,F_X)$ be a triple, where $X$ is a smooth proper variety over a perfect field $k$ of characteristic $p>0$ and $(X,D)$ is an nc pair. Then the triple $(\mathcal{X},\mathcal{D},\widetilde{F}_X)$ is a \textit{quasi-canonical lifting} if $(\mathcal{X},\widetilde{F}_X)$ is a quasi-canonical lifting of $(X,F_X)$ in the sense of Definition \ref{canolift}, and $\mathcal{D} \subseteq \mathcal{X}$ is a divisor with normal crossings relative to $\Spec(W(k))$ such that $\mathcal{D}|_X=D$ and $\widetilde{F}_X^*(\mathcal{D})=p\mathcal{D}$.
\end{definition}

Let $(X,D)$ be an nc pair for a smooth variety over $k$. We define the \textit{logarithmic tangent sheaf} as the subsheaf
$$
T_X(-\log D) \subseteq T_X
$$
that consists of those derivatives that preserve the ideal sheaf $\mathcal{O}_X(-D)$.The \textit{logarithmic cotangent sheaf} $\Omega_X^1(\log D)$ is defined as the dual sheaf of $T_X(-\log D)$. These are locally free sheaves on $X$.

The sheaf $T_X(-\log D)$ can be described as follows. Fix a point $x \in X$ and let $x_1,\ldots,x_n$ be a system of local coordinate functions on $\mathcal{O}_{X,x}$. Assume for simplicity that $D=\{x_1 \cdots x_n=0\}$ (without pulling back from the \'etale coordinate). Then a local basis at $x \in X$ of $T_X(-\log D)$ is given by the set: $x_1 \partial_1,\ldots,x_n \partial_n$, where $\partial_1,\ldots,\partial_n$ is a dual basis of $dx_1,\ldots,dx_n$ of $\Omega^1_X$.

We have the following result which states that a vanishing condition on the tangent sheaf implies the uniqueness of quasi-canonical liftings.

\begin{lemma}
\label{CanonicalUnique}
Let $X$ be a smooth proper variety over a perfect field of characteristic $p>0$. If $X$ admits a quasi-canonical lifting and $H^0(X,T_X \otimes B\Omega_X^1)=0$, then it is canonical. Moreover, if $(X,D)$ is an nc pair admitting a quasi-canonical lifting such that $H^0(X,T_X(-\log D) \otimes B\Omega_X^1)=0$, then it is canonical.
\end{lemma}

\begin{proof}
Let $(\mathcal{X},\widetilde{F}_X)$ be a quasi-canonical lifting of $X$ and let $(\mathcal{X}^{\wedge},\widetilde{F}_X^{\wedge})$ be the $p$-adic completion. Then by \(H^0(X, T_X \otimes B\Omega_X^1) = 0\), $(\mathcal{X}^{\wedge},\widetilde{F}_X^{\wedge})$ is the unique formal lifting of $(X,F_X)$ along $\Spf(W(k))$ in view of \Cref{MSResultsPropositions} (1). But the pair $(\mathcal{X},\widetilde{F}_X)$ is the unique algebraization of the $p$-adic formal scheme $(\mathcal{X}^{\wedge},\widetilde{F}_X^{\wedge})$ by \Cref{IllusieAlgebraization} below. We refer to \cite[Variant 3.3.2]{AWZ21} for the logarithmic case.
\end{proof}

In the above proof, we use the following result on deformation theory which we will use frequently in the rest of this paper.

\begin{theorem}[{\cite[Corollary 8.4.7]{I05}}] 
\label{IllusieAlgebraization}
Let $A$ be a Noetherian ring which is complete with respect to an ideal $I \subsetneq A$. For a scheme $X$ over $Y:=\Spec(A)$, denote by $\widehat{X}$ the $I$-adic formal completion. Let $X \to Y$ be a proper morphism and let $Z \to Y$ be a separated morphism of finite type. Then the correspondence
$$
\Hom_Y(X,Z) \to \Hom_{\widehat{Y}}(\widehat{X},\widehat{Z})~\mbox{defined by}~f \mapsto \widehat{f}
$$
is bijective. In particular, the functor $X \mapsto \widehat{X}$ from the category of proper schemes over $Y$ to the category of proper schemes over $\widehat{Y}$ is fully faithful.
\end{theorem}

We note the following fact.

\begin{lemma}
\label{pullbacknc}
Let $(X,D)$ be an nc pair over $S$ and let $f:Y \to X$ be a smooth morphism. Then $(Y,f^*D)$ is an nc pair over $S$.
\end{lemma}

\begin{proof}
The proof of \citeSta{0CBQ} works in the relative setting.
\end{proof}

For any \(k\)-algebra \(A\), we can define a flat lifting (or quasi-canonical lifting) of \(A\) over \(W(k)\) in the same manner as in \Cref{canolift}, namely, a flat lifting of \(\Spec(A)\) over \(\Spec(W(k))\) (or a quasi-canonical lifting of \(\Spec(A)\) over \(\Spec(W(k))\)) such that they are affine schemes.


The ordinarity condition is necessary for the existence of Frobenius lifting, as the following theorem shows (see \cite[Theorem 1.2]{Nak97}).

\begin{theorem}[Nakkajima]
Let $X$ be a smooth proper scheme over a perfect field \(k\) of characteristic $p>0$. Assume that $X$ admits a smooth flat lifting $\mathcal{X}_1$ over $W_2(k)$ and a lifting $\widetilde{F}_{X,1}:\mathcal{X}_1 \to \mathcal{X}_1$ of the Frobenius morphism $F_X:X \to X$. Then $X$ is ordinary. In particular, if $X$ admits a quasi-canonical lifting, then $X$ is ordinary.    
\end{theorem}

We have a fundamental result on quasi-canonical liftings.

\begin{lemma}
\label{FlatCanonical}
Let $k$ be a perfect field of characteristic $p>0$. Then the Frobenius lifting $\widetilde{F}_k:\Spec(W(k)) \to \Spec(W(k))$ is an isomorphism. Moreover, if $(\mathcal{X}, \widetilde{F}_X)$ is a quasi-canonical lifting of a smooth proper variety $X$ over $k$, then $\widetilde{F}_X:\mathcal{X} \to \mathcal{X}$ is a finite flat morphism and $\mathcal{X}$ is a smooth proper scheme over $W(k)$.
\end{lemma}

\begin{proof}
It is clear that $\widetilde{F}_k$ is bijective, because the associated ring map $W(k) \to W(k)$ is an automorphism which lifts the Frobenius bijection on $k$. Since $\mathcal{X}$ is a proper $W(k)$-scheme by the definition of flat liftings for proper \(k\)-schemes, it follows that the composite morphism $\mathcal{X} \to \Spec(W(k)) \xrightarrow{\widetilde{F}_k} \Spec(W(k))$ is proper. Since morphisms between proper schemes are proper, $\widetilde{F}_X$ is automatically a proper morphism and thus the restricted morphism $F_X:X \to X$ of $\widetilde{F}_X$ is finite since the Frobenius morphism is affine. These assumptions say that \(\widetilde{F}_X\) is also finite by \Cref{FiniteLiftings}.
Finally, since $X$ is a smooth proper $k$-variety, the generic fiber of $\mathcal{X} \to \Spec(W(k))$ is smooth by the theorem of generic smoothness (see \cite[Proposition 3.11]{Sh17} and \Cref{OpenNeighborhoodofClosedFiber}). Then $\mathcal{X}$ is a smooth $W(k)$-scheme. So we find that $\widetilde{F}_X$ is flat by so-called ``miracle flatness'' (\citeSta{00R4}), as desired.
\end{proof}


We will use the following lifting property of finite morphisms.

\begin{lemma} \label{FiniteLiftings}
Let \(\widetilde{f} \colon \mathcal{X} \to \mathcal{Y}\) be a morphism of proper schemes over \(\Spec(V)\), where \((V,sV,k)\) is a discrete valuation ring. Assume that the induced morphism on closed fibers \(\widetilde{f}_s \colon \mathcal{X}_s \to \mathcal{Y}_s\) is finite. Then \(\widetilde{f}\) is finite. If, moreover, \(\mathcal{X}\) and \(\mathcal{Y}\) are flat over \(V\) and \(\widetilde{f}_s\) is \'etale, then \(\widetilde{f}\) is \'etale.
\end{lemma}

\begin{proof}
Since \(\widetilde{f}\) is a morphism between proper schemes, it is a proper morphism. Set \(X \coloneqq \mathcal{X}_s\) and \(Y \coloneqq \mathcal{Y}_s\). Since the restriction morphism \(\widetilde{f}_s\) is assumed to be finite, the dimension of the fiber \((\widetilde{f})^{-1}(y)\) is zero for any \(y \in Y\). Since the dimension of fibers of proper morphisms is upper semi continuous (\citeSta{0D4I}), there exists an open subset \(U\) of \(\mathcal{Y}\) such that \(U\) contains \(Y\) and \(\dim((\widetilde{f})^{-1}(y)) = 0\) for any \(y \in U\). Since \(\mathcal{Y}\) is proper over \(\Spec(V)\), \(U\) should be the whole \(\mathcal{Y}\) by \Cref{OpenNeighborhoodofClosedFiber}. This implies that \(\widetilde{f}\) is finite since proper morphisms with finite fibers are finite (\citeSta{02OG}).

For the \'etale assertion, flatness over \(V\) and the local criterion for flatness show that \(\widetilde{f}\) is flat at every point of \(\mathcal{X}_s\). Since \(\widetilde{f}_s\) is \'etale, \(\widetilde{f}\) is \'etale at every point of \(\mathcal{X}_s\). Its non-\'etale locus is a closed subset of \(\mathcal{X}\), hence proper over \(V\), and has empty special fiber. It is therefore empty, proving that \(\widetilde{f}\) is \'etale. 
This implies that \(\widetilde{f} : \mathcal{X} \to \mathcal{Y}\) is \'etale.
\end{proof}

If \(X\) has a quasi-canonical lifting over \(W(k)\), its Picard group \(\Pic(X)\) can be lifted to specific subgroups of the Picard group \(\Pic(\mathcal{X})\) as follows.

\begin{lemma} \label{LineBundleLift}
Let \(X\) be a normal proper variety over a perfect field \(k\) of characteristic \(p > 0\). Assume that \(X\) admits a quasi-canonical lifting \((\mathcal{X}, \widetilde{F}_X)\) over \(W(k)\). Denote by $\Pic(\mathcal{X}) \to \Pic(X)$ the map of Picard groups induced by the closed immersion $X \hookrightarrow \mathcal{X}$. Set $P\coloneqq \{\mathcal{L} \in \Pic(\mathcal{X})~|~\widetilde{F}_X^*(\mathcal{L}) \cong \mathcal{L}^p\}$. Then $P$ is a subgroup of $\Pic(\mathcal{X})$ and the composite mapping  $P \hookrightarrow \Pic(\mathcal{X}) \to \Pic(X)$ is an isomorphism.
\end{lemma}

\begin{proof}
We have known that \(X\) is globally Frobenius-split by the existence of \((\mathcal{X}, \widetilde{F}_X)\) since the existence of a quasi-canonical lifting (at least, on \(W_2(k)\)) implies Frobenius-splitting (\cite[Theorem 5.5]{Z17}). Since cohomology of coherent sheaves on proper schemes over a Noetherian ring is finitely generated over the base ring (\cite[Corollary 23.18]{GW2}), $H^i(X,\mathcal{O}_X)$ are finite-dimensional $k$-vector spaces and we have that the natural Frobenius action $F^*:H^i(X,\mathcal{O}_X) \to H^i(X,\mathcal{O}_X)$ are bijective for all $i \ge 0$. Then we can apply \Cref{MSResultsPropositions} (2). Namely, let us fix a line bundle $L=L_1 \in \Pic(X_1)$. Noticing that $X_1=X$, we have $F_X^*(L_1) \cong L_1^p$. Then we can get a $p$-adic formal invertible sheaf $\{L_n\}_{n \ge 1}$ such that $L_{n+1}$ is the unique invertible sheaf on $X_{n+1}$ lifting of $L_n$ satisfying $\widetilde{F}_{X,n+1}^*(L_{n+1}) \cong L_{n+1}^p$, where $\widetilde{F}_{X,n+1}:X_{n+1} \to X_{n+1}$ is the reduction of the Frobenius lifting $\widetilde{F}_X$ along $X_{n+1} \hookrightarrow \mathcal{X}$. By the theorem of algebraizing line bundles for proper schemes (\cite[Proposition 24.95]{GW2}), we can find $\mathcal{L} \in \Pic(\mathcal{X})$ in a unique way such that $\mathcal{L}|_X \cong L$ and $\widetilde{F}_X^*(\mathcal{L}) \cong \mathcal{L}^p$. This completes the proof of the lemma.
\end{proof}



\subsection{Extending canonical structures over finite \'etale covers}

We recall the affine form of the Witt vector construction and its relation to \(\delta\)-structures (see \cite{Bor11,BorGur20}).

\begin{definition}
Let \(Z=\Spec(A)\) be an affine scheme. We set \(W_2(Z)\coloneqq\Spec(W_2(A))\). A \(\delta\)-structure on \(A\) is data consisting of a ring homomorphism \(s\colon A\to W_2(A)\) that is a section of the first Witt component \(w_0\colon W_2(A)\to A\).
\end{definition}

\begin{definition}
  Let \(k\) be a perfect field of characteristic \(p > 0\) and let \(X\) be a scheme over \(k\) with the absolute Frobenius morphism \(F_X: X \to X\).
  \begin{enumerate}
    \item A \emph{\(p\)-adic formal lifting} \(\{(X_n, F_{X, n})\}_{n \geq 1}\) of \(X\) over \(W(k)\) is a pair of a \(p\)-adic formal scheme \(\{X_n\}_{n \geq 1}\) over \(W(k)\) (\Cref{DefFormalSchemes}) and a compatible system of Frobenius liftings \(\{F_{X, n}\}_{n \geq 1}\) on \(X_n\) such that \(X_1 \cong X\) and \(F_{X, 1} \cong F_X\).
    \item   Let \(f \colon Y \to X\) be a morphism of schemes over \(k\) with the absolute Frobenius morphism \(F_Y: Y \to Y\) and let \(\{(Y_n, F_{Y, n})\}_{n \geq 1}\) be a \(p\)-adic formal lifting of \(Y\) over \(W(k)\).
  A morphism of the \(p\)-adic formal liftings of \(Y\) and \(X\) is a morphism of \(p\)-adic formal schemes \(\{f_n\}_{n \geq 1} : \{Y_n\}_{n \geq 1} \to \{X_n\}_{n \geq 1}\) the following diagram commutes for all \(n \geq 1\):
  \begin{equation*}
    \xymatrix@C=3.5em@R=2.2em{
      Y_n \ar[r]^{f_n} & X_n & & Y_n \ar[d]_{f_n} \ar[r]^{F_{Y,n}} & Y_n \ar[d]^{f_n} \\
      Y \ar[u] \ar[r]_{f} & X \ar[u] & & X_n \ar[r]_{F_{X,n}} & X_n.
    }
  \end{equation*}
  \item An \emph{algebraization} of a \(p\)-adic formal lifting \(\{(X_n, F_{X, n})\}_{n \geq 1}\) of \(X\) is a quasi-canonical lifting \((\mathcal{X}, \widetilde{F}_X)\) of \(X\) which has an isomorphism of \(p\)-adic formal liftings \(\{(X_n, F_{X, n})\}_{n \geq 1} \cong \{(\mathcal{X}_n, \widetilde{F}_{X, n})\}_{n \geq 1}\), where \(\mathcal{X}_n\) is the reduction of \(\mathcal{X}\) along \(W(k) \twoheadrightarrow W_n(k)\) and \(\widetilde{F}_{X, n}\) is the reduction of \(\widetilde{F}_X\) along \(\mcalX_n \hookrightarrow \mcalX\).
  \end{enumerate}
\end{definition}

Lifting properties can be inherited under finite \'etale covers, as the following proposition shows. In the special case of ordinary varieties with trivial cotangent bundle, preservation of ordinarity under finite \'etale covers is proved in \cite[Lemma 1.2]{MS87}.

\begin{proposition}
\label{liftcanonical}
Let $X$ and \(Y\) be projective varieties over an algebraically closed field $k$ of characteristic $p>0$. Assume that $f:Y \to X$ is a surjective finite \'etale morphism.
Then the following assertions hold.
\begin{enumerate}
  \item If there exists a \(p\)-adic formal scheme \(\{X_n\}_{n \geq 1}\) over \(W(k)\) such that \(X_1 \cong X\), then there exists a unique morphism \(\{f_n\}_{n \geq 1} : \{Y_n\}_{n \geq 1} \to \{X_n\}_{n \geq 1}\) of \(p\)-adic formal schemes such that \(f_1\) is the given morphism \(f\) and each \(f_n\) is a surjective finite \'etale morphism.
  \item If \(X\) has a \(p\)-adic formal lifting \(\{(X_n, F_{X, n})\}_{n \geq 1}\) over \(W(k)\), then there exists a unique morphism \(\{f_n\}_{n \geq 1} \colon \{(Y_n, F_{Y, n})\}_{n \geq 1} \to \{(X_n, F_{X, n})\}_{n \geq 1}\) of \(p\)-adic formal liftings such that \(f_1\) is the given morphism \(f\) and each \(f_n\) is a surjective finite \'etale morphism.
  \item If \(X\) has a projective flat lifting \(\mathcal{X}\) over \(W(k)\), then there exists a unique finite \'etale surjective morphism \(\widetilde{f} \colon \mathcal{Y} \to \mathcal{X}\) of projective flat schemes over \(W(k)\) such that the reduction of \(\widetilde{f}\) along \(\Spec(k) \hookrightarrow \Spec(W(k))\) is the given morphism \(f\).
  \item If \(X\) has a projective quasi-canonical lifting \((\mathcal{X}, \widetilde{F}_X)\) over \(W(k)\), then there exists a unique finite \'etale surjective morphism \(\widetilde{f} \colon (\mathcal{Y}, \widetilde{F}_Y) \to (\mathcal{X}, \widetilde{F}_X)\) of projective quasi-canonical liftings such that the reduction of \(\widetilde{f}\) along \(\Spec(k) \hookrightarrow \Spec(W(k))\) is the given morphism \(f\).
\end{enumerate}

\end{proposition}

\begin{proof}
(1): Since $Y \to X$ is \'etale, it follows from \cite[Tag \href{https://stacks.math.columbia.edu/tag/08R2}{08R2} and Tag \href{https://stacks.math.columbia.edu/tag/08T3}{08T3}]{Stacks} that $\mathbb{L}_{Y/X} \simeq_{q.i.} 0$. We want to construct a unique family of surjective finite \'etale morphisms $f_n \colon Y_n \to X_n$ such that $Y_n$ is flat over $W_n(k)$, fitting into the commutative diagram with cartesian squares:
\[
\xymatrix@M=10pt{
Y_1\ar[r] \ar[d]^f   & \ar@{.}[r] & \ar[r] & Y_n \ar[r]\ar[d]^{f_n} & ? \ar[d] \\
X_1\ar[r] & \ar@{.}[r] & \ar[r] & X_n \ar[r] & X_{n+1} \\
}
\]
Suppose that we have constructed the desired $f_n:Y_n \to X_n$.

Let us construct $Y_{n+1}$. By \'etaleness of $f_n:Y_n \to X_n$, we have $\mathbb{L}_{Y_n/X_n} \simeq_{q.i.} 0$, which gives a morphism $f_{n+1}:Y_{n+1} \to X_{n+1}$ uniquely since the obstruction of such lifting belongs to a second \(\Ext\) from \(\mathbb{L}_{Y_n/X_n}\) (\cite[Theorem A.4]{Z17}, \citeSta{08UZ} or \cite[Th\'eor\`eme 2.1.7 (ii)]{Ill1}) where $\mathcal{O}_{Y_{n+1}}$ is flat over \(W_{n+1}\) and is defined as an extension of the sheaf of rings $\mathcal{O}_{Y_n}$ by a square-zero ideal. By applying the preservation of finite \'etaleness under square-zero extension (\citeSta{06AG}), we find that $f_{n+1}$ is finite \'etale, because $f_n$ is so. In particular, $\mathbb{L}_{Y_{n+1}/X_{n+1}} \simeq_{q.i.} 0$. This proves (1).

(2): We have already constructed a unique family of surjective finite \'etale morphisms $f_n:Y_n \to X_n$ such that $Y_n$ is flat over $W_n(k)$.
To construct \(F_{Y,n+1}\), consider the commutative diagram
\begin{equation*}
    \begin{tikzcd}[ampersand replacement=\&]
    Y_n \arrow[d, hook] \arrow[r, "{F_{Y, n}}"] \& Y_n \arrow[r, hook]               \& Y_{n+1} \arrow[d, "f_{n+1}"] \\
    Y_{n+1} \arrow[r, "f_{n+1}"]                \& X_{n+1} \arrow[r, "{F_{X, n+1}}"] \& X_{n+1},                     
    \end{tikzcd}
\end{equation*}
where \(Y_n \hookrightarrow Y_{n+1}\) is the closed immersion. Since \(f_{n+1}\) is \'etale, there is a unique morphism \(F_{Y,n+1} \colon Y_{n+1} \to Y_{n+1}\) extending \(F_{Y,n}\) and satisfying \(f_{n+1}\circ F_{Y,n+1}=F_{X,n+1}\circ f_{n+1}\). Starting with \(F_{Y,1}=F_Y\), this proves (2). 

(3): Next, we assume that \(X\) has a projective flat lifting \(\mcalX\) over \(W(k)\). Set \(X_n \coloneqq \mcalX \times_{\Spec(W(k))} \Spec(W_n(k))\) for \(n \geq 1\). Then we have a \(p\)-adic formal scheme \(\{X_n\}_{n \geq 1}\) over \(W(k)\) such that \(X_1 \cong X\).
We already have a finite \'etale surjective morphism \(\{f_n\}_{n \geq 1} \colon \{Y_n\}_{n \geq 1} \to \{X_n\}_{n \geq 1}\) by (1). Let \(\mathcal{L}\) be an ample line bundle relative to \(\mathcal{X} \to \Spec(W(k))\). There is a system $\{X_n,L_n\}_{n \ge 1}$ such that $X_n$ (resp. $L_n$) is the reduction of $\mathcal{X}$ (resp. $\mathcal{L}$) along $\Spec(W_n(k)) \to \Spec(W(k))$. So \(X_n\) is a flat proper \(W_n(k)\)-scheme and \(L_n\) is an ample line bundle over \(X_n\) since pullback of ample line bundles by a quasi-affine morphism is again ample line bundles (\citeSta{0892}). Let $L'_n \coloneqq f_n^*L_n$ be the pullback of $L_n$ along $f_n$. Then $L'_n$ is an ample line bundle because $f_n$ is a surjective finite (\'etale) surjective since pullback of ample line bundle by a proper, finite, and surjective morphism is again ample line bundle (\citeSta{0B5V}). Now the unique $p$-adic formal scheme $\{Y_n,L'_n\}_{n \ge 1}$ admits an algebraization $\mathcal{Y}$ and \(\mathcal{L}'\) where \(\mathcal{Y}\) is a projective scheme over $\Spec(W(k))$ and \(\mathcal{L}'\) is an ample line bundle on \(\mathcal{Y}\) by Grothendieck's algebraization theorem (\citeSta{089A}). Moreover, $f:Y \to X$ lifts uniquely to a morphism $\widetilde{f}:\mathcal{Y} \to \mathcal{X}$ of proper \(W(k)\)-schemes by \citeSta{0A42} (or \Cref{IllusieAlgebraization}). In other words, $\widetilde{f}$ restricts to the finite \'etale morphism $f \colon Y \to X$. By \Cref{FiniteLiftings}, the lifting $\widetilde{f}$ is finite \'etale. This is an open map and thus $\widetilde{f}$ is a surjective finite \'etale morphism by \Cref{OpenNeighborhoodofClosedFiber}. This proves that $\{f_n\}_{n \ge 1}$ constructed in (1) is algebraizable.

(4): Finally, we assume that \(X\) has a projective quasi-canonical lifting \((\mathcal{X}, \widetilde{F}_X)\) and let \(\{(X_n, F_{X, n})\}_{n \geq 1}\) be the \(p\)-adic formal lifting of \((X, F_X)\) over \(W(k)\) induced by \((\mathcal{X}, \widetilde{F}_X)\). By (2) and (3), it remains to algebraize the compatible Frobenius liftings \(F_{Y,n}\). Applying \Cref{IllusieAlgebraization} with base \(\Spec(W(k))\), source \(\mathcal{Y}\), and target \(\mathcal{Y}^{(1)}\), where \(\mathcal{Y}^{(1)} \defeq \mathcal{Y} \times_{\Spec(W(k)), F} \Spec(W(k))\), we obtain a unique Frobenius lifting \(\widetilde{F}_Y\colon\mathcal{Y}\to\mathcal{Y}\). Applying the same theorem with source \(\mathcal{Y}\) and target \(\mathcal{X}^{(1)}\), the identities \(f_n\circ F_{Y,n}=F_{X,n}\circ f_n\) for all \(n\geq1\) imply \(\widetilde{f}\circ\widetilde{F}_Y=\widetilde{F}_X\circ\widetilde{f}\). Thus \((\mathcal{Y},\widetilde{F}_Y)\) is the required quasi-canonical lifting.
\end{proof}

\begin{remark}
The implication in \Cref{liftcanonical}(4) asserts that a Frobenius lifting on \(\mcalX\) induces one on a finite \'etale cover \(\mcalY\).
Suppose that \(k = \overline{\setF_p}\) and that \(f \colon \mcalY \to \mcalX\) is the \(p\)-adic completion of an \'etale morphism between finite-dimensional smooth \(W(k)\)-schemes.
In this case, the implication can be explained using arithmetic jet spaces. We thank the referee for suggesting this argument.

Equip \(W(k)\) with its canonical \(\delta\)-structure, and let \(J^1\) denote Buium's first arithmetic jet space in the category of \(p\)-adic formal schemes (\cite{Buium1996Geometry}). For a smooth \(p\)-adic formal \(W(k)\)-scheme \(\mcalZ\), sections of \(J^1(\mcalZ) \to \mcalZ\) correspond to \(\delta\)-structures on \(\mcalO_{\mcalZ}\) extending that on \(W(k)\); see \cite[\S 1.0.13, Proposition 1.0.6]{DupuyThesis}.

For an \'etale morphism \(f\colon\mcalY\to\mcalX\), the canonical morphism \(J^1(\mcalY) \to \mcalY \times_{\mcalX} J^1(\mcalX)\) is an isomorphism; see \cite[\S 3.2]{Buium2005Arithmetic} and \cite[Lemma 4.2(1)]{DupuyLiftedTorsors}. Thus the section \(s_X\colon\mcalX\to J^1(\mcalX)\) corresponding to \(\widetilde{F}_X\) induces a section
\[
s_Y=(\id_{\mcalY}, s_X \circ f) \colon \mcalY \to \mcalY \times_{\mcalX} J^1(\mcalX) \xleftarrow{\cong} J^1(\mcalY)
\]
of \(J^1(\mcalY) \to \mcalY\).
The corresponding Frobenius lifting \(\widetilde{F}_Y\) satisfies \(f \circ \widetilde{F}_Y = \widetilde{F}_X \circ f\).
\end{remark}

We have a weak analogue of \Cref{liftcanonical} in the logarithmic setting.

\begin{corollary}
\label{liftcanonical2}
Let $(X,D)$ be a smooth projective nc pair defined over an algebraically closed field $k$ of characteristic $p>0$. Assume that $(\mathcal{X},\mathcal{D},\widetilde{F}_X)$ is a quasi-canonical lifting of $(X,D,F_X)$. Let $f:(Y,D_Y) \to (X,D)$ be a surjective finite \'etale morphism with $D_Y \defeq f^{*}D$. Let $\widetilde{f}:\mathcal{Y} \to \mathcal{X}$ be the morphism constructed in \Cref{liftcanonical} (4) and let $\mathcal{D}_Y \defeq \widetilde{f}^*\mathcal{D}$. Then $(\mathcal{Y},\mathcal{D}_Y,\widetilde{F}_Y)$ is a quasi-canonical lifting of $(Y,D_Y,F_Y)$.
\end{corollary}

\begin{proof}
By \Cref{liftcanonical}, since $\widetilde{f}$ is finite \'etale, we see from Lemma \ref{pullbacknc} that $(\mathcal{Y},\mathcal{D}_Y)$ is an nc pair relative to $S=\Spec(W(k))$. So it remains to show that $\widetilde{F}_Y^*\mathcal{D}_Y=p\mathcal{D}_Y$. Again by \Cref{liftcanonical}, we have $\widetilde{f} \circ \widetilde{F}_Y=\widetilde{F}_X\circ \widetilde{f}$. We pull back $\mathcal{D}$ in two different ways. First, we get
$$
(\widetilde{F}_X \circ \widetilde{f})^*\mathcal{D}=\widetilde{f}^*(\widetilde{F}_X^*\mathcal{D})=\widetilde{f}^*(p\mathcal{D})=p\mathcal{D}_Y.
$$
On the other hand,
$$
(\widetilde{f} \circ \widetilde{F}_Y)^* \mathcal{D}= \widetilde{F}_Y^*(\widetilde{f}^*\mathcal{D})=\widetilde{F}_Y^*\mathcal{D}_Y.
$$
Now we complete the proof.
\end{proof}


\section{A construction of Frobenius lifts via geometric methods}

We will show our one of the main results \Cref{algquotient}.

\subsection{Finite \'etale quotients and fppf sheaves}
In this subsection, we will review some basic facts on finite \'etale quotients and fppf sheaves.
This will be used to construct liftings from ones of finite \'etale covers and its equivalence relations.

Let $S$ be a base scheme and let $f:X \to Y$ be a surjective \'etale morphism of $S$-schemes. Set $R \defeq X \times_Y X$ and $j:R \to X \times_S X$ be the natural morphism and let $p_i:X \times_S X  \to X$ ($i=1,2$) be the projection into the $i$-th factor. Then $j$ defines an \'etale equivalence relation (see \citeSta{022P} for relevant notions). Let $\mathrm{Sh}_{fppf}({\mathrm{Sch}}/S)$ be the category of sheaves of sets on the category of S-schemes with respect to fppf topology. Set $\sigma_i \defeq p_i \circ j$.
Then it will be convenient to interpret the quotient $X \to Y$ as the coequalizer of the diagram: $\sigma_1,\sigma_2:R \rightrightarrows X$ (\Cref{algspacepre} below) since some lifting problems can be reduced to the case of finite \'etale covers and its equivalence relations.
Let $F:({\mathrm{Sch}}/S)_{fppf} \to \mathrm{Sets}$ be the coequalizer of this diagram. Equivalently, $F$ is the sheafification of the correspondence $U \in ({\mathrm{Sch}}/S)_{fppf} \mapsto X(U)/R(U)$. Then the following result on the presentation of the functor $F$ is essential.

\begin{proposition}
\label{algspacepre}
Let the notation and the hypotheses be as above. Namely, let \(S\) be a scheme and let \(f:X \to Y\) be a surjective \'etale morphism of \(S\)-schemes. Set \(R  \defeq  X \times_Y X\) and the projections $\sigma_1, \sigma_2: R \rightrightarrows X$. Then $\sigma_1,\sigma_2$ are \'etale morphisms and the diagram $R \rightrightarrows X \to Y$ is a coequalizer diagram in the category of sheaves on $({\mathrm{Sch}}/S)_{fppf}$. If moreover $f:X \to Y$ is finite, then $\sigma_1,\sigma_2$ are finite.
\end{proposition}

We refer the reader to \citeSta{0262} for the proof of \Cref{algspacepre}. Under the above hypotheses, $\sigma_1,\sigma_2$ are \'etale morphisms, so the method of formal deformation theory via the cotangent complex works nicely for our purpose. Although we do not need algebraic spaces, we hope to generalize the main results of this article to the setting of deformation of formal algebraic spaces.
Such a study is essential in view of Matsumoto's example of a K3 surface having an integral model with good reduction only in the category of algebraic spaces. See \cite[Example 5.2]{Matsumoto15} for details.

Now let us turn our attention to the basic setting.
We will be interested in the case where $X \to Y$ is a surjective finite \'etale morphism between projective varieties (see \Cref{HironakaExamp} below).
We need another lemma on the characterization of sheaf topoi (known as Giraud's axioms).

\begin{lemma}
\label{Giraud}
Let $S$ be a scheme and let $R \rightrightarrows U \to X$ be a coequalizer in the category of sheaves of sets on $({\mathrm{Sch}}/S)_{fppf}$. For a scheme map $S' \to S$, the induced diagram
$$
R \times_S S' \rightrightarrows U \times_S S' \to X \times_S S'
$$
is a coequalizer in the category of sheaves of sets on $({\mathrm{Sch}}/S')_{fppf}$.
\end{lemma}

\begin{proof}
This follows from the exactness of the base change functor \(- \times_S S' : ({\mathrm{Sch}}/S)_{fppf} \to ({\mathrm{Sch}}/S')_{fppf}\) (see \citeSta{0DTF}). Another exposition can be found in \cite[Proposition 6.1.0.1]{L09}.\footnote{We use this result only for calculating coequalizer with finite \'etale equivalences. However, if the equivalence relation in consideration is not \'etale, the resulting coequalizer in the category of sheaves with respect to some topology is, in general, different from the coequalizer in the category of schemes. See also \Cref{HironakaExamp}.}   
\end{proof}

We recall a result of Altman-Kleiman \cite{AK80} which ensures the quasi-projectivity of quotients. We follow the statement written in \cite{N05}*{Theorem 5.25}.

\begin{lemma} \label{AltmanKleiman}
    Let \(S\) be a Noetherian scheme and let \(X\) be a quasi-projective \(S\)-scheme.
    Let \(\sigma \colon R \to X \times_S X\) be an equivalence relation (\citeSta{022P}) of schemes on \(X\) over \(S\) with proper and flat projections \(\sigma_i \colon R \to X\).
    Then the quotient \(X \to X/R\) in \((\mathrm{Sch}/S)_{fppf}\) can be represented by an \(S\)-scheme morphism \(X \to Q\) such that \(Q\) is quasi-projective over \(S\), the morphism \(X \to Q\) is faithfully flat and projective, and the induced morphism \(R \to X \times_Q X\) is an isomorphism.
\end{lemma}

\subsection{Algebraization of formal schemes}

We need a lemma on the trace map for finite flat morphisms.

\begin{lemma}
\label{tracemap}
Let \(f \colon X \to Y\) be a surjective finite morphism of integral Noetherian schemes.
Assume that \(Y\) is normal and let \(\mathrm{Tr}_{K(X)/K(Y)} \colon K(X) \to K(Y)\) be the usual trace map of the finite field extension \(K(X)/K(Y)\) (see \citeSta{0BIE}).
Then \(\mathrm{Tr}_{K(X)/K(Y)}\) induces a map of $\mathcal{O}_Y$-modules: \(\mathrm{Tr}_f \colon f_*\mathcal{O}_X \to \mathcal{O}_Y\) under the natural inclusions \(\Gamma(f^{-1}(U), \mathcal{O}_X) \subseteq K(X)\) and \(\Gamma(U, \mathcal{O}_Y) \subseteq K(Y)\) for any affine open \(U \subseteq Y\).
Suppose further that $d \coloneqq [K(X):K(Y)]$ is invertible in $\Gamma(Y,\mathcal{O}_Y)$. Then the normalized trace $\frac{1}{d}\mathrm{Tr}_f: f_*\mathcal{O}_X \to \mathcal{O}_Y$ is a splitting map of the natural map $f^\sharp: \mathcal{O}_Y \to f_*\mathcal{O}_X$. 
\end{lemma}

\begin{proof}
It suffices to check that for an open affine $U \subseteq Y$, the map $\mathrm{Tr}_f(U)$ is well-defined. Since $f$ is an affine morphism, $f^{-1}(U)$ is affine. So we may assume that $U \defeq \Spec(R)$ is a Noetherian normal domain and $R \to S \defeq  \Gamma(f^{-1}(U), \mathcal{O}_X)$ is a module-finite extension. Moreover, we have $\Frac(R)=K(Y)$ and $\Frac(S)=K(X)$. Let $x \in S$ and let $T^n + a_1T^{n-1} + \cdots + a_n \in K(Y)[T]$ be the minimal polynomial of $x$ over $K(Y)$ with \(en = [K(X): K(Y)]\). Then we need to show that $\mathrm{Tr}_{K(X)/K(Y)}(x) = -ea_1 \in R$ by \citeSta{0BIH}. Since $R$ is integrally closed in $K(Y)$ and any coefficients of minimal polynomials of \(x \in S \subseteq K(X)\) over \(K(Y)\) are contained in the integral closure of \(R\) (\cite[Theorem 2.1.17]{SwHu}), it follows that $a_1 \in R$. Now we see that the restriction of the $K(Y)$-module map \(\mathrm{Tr}_{K(X)/K(Y)}\) to $U=\Spec(R)$ is the $R$-module map \(\mathrm{Tr}_f(U)\), so we get a well-defined map \(\mathrm{Tr}_f\). 

If we take \(x \in \Gamma(U,\mathcal{O}_Y)=R\subseteq S\), then its minimal polynomial over \(K(Y)\) is \(T-x\), and hence \(\mathrm{Tr}_{K(X)/K(Y)}(x)=[K(X):K(Y)]x\). This shows that the normalized trace map is a splitting map of the natural map \(f^\sharp\).
\end{proof}

The following lemma is a special case of \cite[XVIII Théorème 2.9(Var 4)(I)]{SGAIV} in which \(f\) is more general.
However, for the convenience for the readers, we write the proof of this special case.

\begin{lemma}
\label{TraceTimesd}
Let \(f \colon X \to Y\) be a surjective finite \'etale morphism of constant degree \(d\) between Noetherian schemes. Then the \emph{trace map} \(\mathrm{Tr}_f \colon f_*\mathcal{O}_X \to \mathcal{O}_Y\) exists which is a map of $\mathcal{O}_Y$-modules and the composition
\begin{equation*}
\mathcal{O}_Y \xrightarrow{f^\sharp} f_*\mathcal{O}_X \xrightarrow{\mathrm{Tr}_f} \mathcal{O}_Y
\end{equation*}
is multiplication by \(d\).
\end{lemma}

\begin{proof}
In this setting, the derived pushforward functor \(Rf_* \colon D_{\mathrm{q.coh}}(X) \to D_{\mathrm{q.coh}}(Y)\) has the right adjoint functor \(f^\times\). This notation is based on \cite[Theorem 25.17]{GW2}. Since \(f\) is proper and \'etale, $f^! = f^\times = f^*$ by \cite[Theorem and Definition 25.61 and Corollary 25.69]{GW2}. Therefore, \(f^{\times}\mathcal{O}_{Y} \simeq_{q.i.} \mathcal{O}_{X}\). Note that any affine morphism has the vanishing of the higher direct images (see, for example, \citeSta{01XC}). It follows from \cite[(25.3.2)]{GW2} that there is a well-defined \emph{trace} ${\mathrm{Tr}_f}:f_*\mathcal{O}_{X} = Rf_*f^\times\mathcal{O}_Y \to \mathcal{O}_{Y}$. By \cite[Reminder 25.24 and Theorem 25.31]{GW2}, \(f^\times\) (and thus \(\mathrm{Tr}_f\)) commutes with any pullback of flat morphisms \(U \to Y\). So we can assume that \(f\) is the canonical map \(X = \bigsqcup_{i=1}^d Y \to Y\) since \(f\) is of constant degree \(d\). For each \(i = 1, \dots, d\), the inclusion \(\iota_i \colon Y \hookrightarrow \bigsqcup_{i=1}^d Y\) makes a commutative diagram
\begin{center}
\begin{tikzcd}
{f_i}_* \circ f_i^\times \arrow[d, "\mathrm{Tr}_{f_i}"] \arrow[r, Rightarrow, no head] & f_* \circ {\iota_i}_* \circ {\iota_i}^\times \circ f^\times \arrow[d, "\mathrm{Tr}_{\iota_i}"] \\
\id                                                                                             & f_* \circ f^\times \arrow[l, "\mathrm{Tr}_f"]                                        
\end{tikzcd}
\end{center}
with \(f_i \coloneqq f \circ \iota_i\) by applying \cite[Proposition 25.19 and Proposition 21.41]{GW2} and the adjunction \(f_* \dashv f^! = f^*\). Indeed the composition \(f_* \circ {\iota_i}_* \circ \iota_i^\times \circ f^\times \to f_* \circ f^\times \to \id\) corresponds to \({\iota_i}_* \circ \iota_i^\times \circ f^\times \to f^\times \xrightarrow{\id} f^\times\) by the adjunction \(f_* \dashv f^\times\) and this corresponds to the identity map on \(f_i^\times = \iota_i^\times \circ f^\times\) by the adjunction \({\iota_i}_* \dashv \iota_i^\times\). Here, \(\mathrm{Tr}_{f_i}\) also corresponds to this identity map by the adjunction \({f_i}_* \dashv f_i^\times\) and the above diagram is commutative.
  The map \(f_i \colon Y \to Y\) is the identity map \(\id_{Y}\) and the source \(f_*f^*\mathcal{O}_{Y} = f_*\mathcal{O}_X\) of \(\mathrm{Tr}_f\) is \(\prod_{i=1}^d \mathcal{O}_{Y}\).
  So the trace map \(\mathrm{Tr}_f \colon f_*\mathcal{O}_{X} \to \mathcal{O}_{Y}\) is given by the summation map \(\oplus_{i=1}^d \mathcal{O}_{Y} \to \mathcal{O}_{Y}\). This shows that the composite mapping
$$
\mathcal{O}_{Y} \xrightarrow{f^\sharp} f_* \mathcal{O}_{X} \xrightarrow{\mathrm{Tr}_f} \mathcal{O}_{Y} 
$$
is multiplication by $d$.
\end{proof}

Let us prove the first main result, which will be quite useful for constructing new classes of projective varieties over a field of characteristic $p>0$ with a quasi-canonical lifting. Notice that there is a classical result opposite to what is to be proven below (see \citeSta{09ZT}).

\begin{proof}[Proof of Main Theorem \ref{algquotient}]
$(1)$: We aim to construct a tower $\{Y_n\}_{n \ge 1}$ inductively. Fix $n \ge 1$. Suppose that the pullback diagram
$$
\vcenter{
\xymatrix@M=10pt{ 
Z_{j-1} \ar[d]_{f_{j-1}} \ar[r] &Z_j  \ar[d]_{f_j}\\
Y_{j-1} \ar[r] &Y_j \\
}}
$$
satisfying the properties $(a)$ and $(b)$ of Proof of Main Theorem \ref{algquotient} has been constructed for all $j \leq n$. By assumption, there is a flat $S_{n+1}$-scheme $Z'_{n+1}$ and a morphism $Z_n \to Z'_{n+1}$ such that $Z_n \cong Z'_{n+1} \times_{S_{n+1}} S_n$.
Set \(I_n\coloneqq p^nW_{n+1}(k)=\ker(W_{n+1}(k)\to W_n(k))\), regarded as a \(W_n(k)\)-module, and write \(J_{Y_n}\coloneqq\mathcal{O}_{Y_n}\otimes_{W_n(k)}I_n\) and \(J_{Z_n}\coloneqq\mathcal{O}_{Z_n}\otimes_{W_n(k)}I_n \cong f_n^*J_{Y_n}\).
Take a map of deformation tuples (see \cite[Definition A.3 and Definition A.5]{Z17} for the definition of deformation tuples and their maps)
\begin{equation}
\label{commcotangent1}
  \begin{tikzcd}
    Z_n \arrow[r] \arrow[d, "a_{Z_n}"]      &    Z'_{n+1} \arrow[d]                \\
    S_n \arrow[r] \arrow[d] & S_{n+1} \arrow[ld] \\
    S_{n+1}                      &                   
  \end{tikzcd} \xrightarrow{{\text{induced by}~f_n}} \begin{tikzcd}
    Y_n \arrow[d, "a_{Y_n}"]      &                    \\
    S_n \arrow[r] \arrow[d] & S_{n+1} \arrow[ld] \\
    S_{n+1}                    &                   
  \end{tikzcd}
\end{equation}
with the trivial morphism \(a_{Z_n}^*I_n \xrightarrow{\id} J_{Z_n}\) of \(\mathcal{O}_{Z_n}\)-modules, \(a_{Y_n}^*I_n \xrightarrow{\id} J_{Y_n}\) of \(\mathcal{O}_{Y_n}\)-modules, and the morphism \(\eta \colon J_{Y_n} \to Rf_{n, *}Lf_{n}^*J_{Y_n} \cong Rf_{n, *}J_{Z_n}\).
In the derived category $D_{\mathrm{q.coh}}(Y_n)$, the morphism of deformation tuples induces the following commutative diagram (see \cite[Lemma A.6 and Lemma 3.1]{Z17}) associated to $(\ref{commcotangent1})$:
\begin{equation*}
    \begin{tikzcd}[ampersand replacement=\&]
    {\mathbb{L}_{Y_n/S_n}} \arrow[rr, "\ob(Y_n)"] \arrow[d, "df_n"] \& \& {J_{Y_n}[2]} \arrow[d, "{\eta[2]}"] \\
    {Rf_{n*}\mathbb{L}_{Z_n/S_n}} \arrow[rr, "{Rf_{n*}\ob(Z_n)}"] \& \& {Rf_{n*}J_{Z_n}[2]}.
    \end{tikzcd}
\end{equation*}
The obstruction to lifting \(Y_n\) to \(S_{n+1}\) is \(\ob(Y_n)\in\Ext^2_{Y_n}(\mathbb{L}_{Y_n/S_n},J_{Y_n})\). Since the lifting \(Z'_{n+1}\) exists, we have \(\ob(Z_n)=0\). The morphism \(\eta\colon J_{Y_n}\to Rf_{n*}J_{Z_n}=f_{n*}f_n^*J_{Y_n}\) splits: by \Cref{TraceTimesd} and the projection formula, the composite
\[
J_{Y_n}\to f_{n*}f_n^*J_{Y_n}\cong Rf_{n*}\mathcal{O}_{Z_n}\otimes_{\mathcal{O}_{Y_n}}J_{Y_n}\xrightarrow{d^{-1}\mathrm{Tr}_{f_n}\otimes\mathrm{id}}J_{Y_n}
\]
is the identity. The commutative diagram above therefore gives \(\ob(Y_n)=0\). Consequently, there exists a flat \(S_{n+1}\)-scheme \(Y_{n+1}\) fitting into a commutative diagram
\begin{center}
  \begin{tikzcd}
    Y_n \arrow[d] \arrow[r] & Y_{n+1} \arrow[d]  \\
    S_n \arrow[r] \arrow[d] & S_{n+1} \arrow[ld] \\
    S_{n+1}                    &                   
  \end{tikzcd}
\end{center}
such that \(Y_n \cong Y_{n+1} \times_{S_{n+1}} S_n\) (\cite[Theorem A.4]{Z17} or \cite[Th\'eor\`eme 2.1.7 (ii)]{Ill1}). Now the existence of $Y_{n+1}$ is established.\footnote{There may not exist a morphism from $Z'_{n+1}$ to $Y_{n+1}$.}
Take a deformation tuple as in \cite[Definition A.3]{Z17}:
\begin{center}
  \begin{tikzcd}
    Z_n \arrow[d, "f_n"]      &                    \\
    Y_n \arrow[r] \arrow[d] &Y_{n+1} \arrow[ld] \\
    S_{n+1}                      &                   
  \end{tikzcd}
\end{center}
with the canonical isomorphism \(f_n^*J_{Y_n}\xrightarrow{\sim}J_{Z_n}\) of \(\mathcal{O}_{Z_n}\)-modules.

The \'etaleness of \(f_n\) implies the vanishing of \(\mathbb{L}_{Z_n/Y_n}\) and the same holds for the obstruction of the lifting of \(Y_n \to Y_{n+1}\) by \cite[Theorem A.4]{Z17}.
Therefore, there is $Z_{n+1}$\footnote{It may differ from $Z'_{n+1}$ in general.} so that we have a pullback diagram
$$
\vcenter{
\xymatrix@M=10pt{ 
Z_{n} \ar[d]_{f_{n}} \ar[r] &Z_{n+1}  \ar[d]_{f_{n+1}}\\
Y_{n} \ar[r] &Y_{n+1} \\
}}
$$
where $Z_{n+1}$ is flat over $S_{n+1}$ and $Z_n \cong Z_{n+1} \times_{S_{n+1}} S_n$: the last isomorphism follows from \(Y_{n+1} \times_{S_{n+1}} S_n \cong Y_n\).
Moreover, we have $\mathbb{L}_{Z_n/Y_n} \simeq_{q.i.} 0$ by \'etaleness of $f_n$. Now \citeSta{08UZ} implies the uniqueness of $f_{n+1}:Z_{n+1} \to Y_{n+1}$ as a (flat) lifting of $f_n$. It follows from \citeSta{06AG} that $f_{n+1}$ is finite \'etale since \(Z_{n+1}\) is flat over \(S_{n+1}\), \(f_n\) is finite \'etale, and \(Z_n \to Z_{n+1}\) and \(Y_n \to Y_{n+1}\) are first order thickenings.
Since the reduced part of $f_{n+1}$ is identified with $f$, it follows that $f_{n+1}$ is surjective and has constant degree $d$. Hence we have the desired \(p\)-adic formal scheme \(\{Y_n\}_{n \geq 1}\) and the morphism \(\{f_n\}_{n \geq 1}\) satisfying the properties $(a)$ and $(b)$.

$(2)$: Assume that $\mathcal{Z}$ is an algebraization of the $p$-adic formal scheme $\{Z_n\}_{n \ge 0}$ which is a flat projective \(W(k)\)-scheme. Then there is an ample line bundle $\mathcal{L}$ over $\mathcal{Z}$. Write $L_n$ for the pullback of $\mathcal{L}$ alo ng $Z_n \hookrightarrow \mathcal{Z}$ which is also an ample line bundle over \(Z_n\) by \citeSta{0892}.
Now consider $\{Y_n\}_{n \ge 1}$.

Let us check that $Y_n$ is a flat projective $S_n$-scheme. Flatness was already verified above. Here is one way to see the projectivity. Set $R_n \defeq Z_n \times_{Y_n} Z_n$. Then we have a morphism $R_n \to Z_n \times_{S_n} Z_n$. By composing this with the projections $p_1,p_2: Z_n \times_{S_n} Z_n \rightrightarrows Z_n$, we get $\sigma_{1},\sigma_{2}:R_n \rightrightarrows Z_n$.
Since \(f_n \colon Z_n \to Y_n\) is a finite \'etale surjection of \(S_n\)-schemes,
the resulting quotient sheaf $Z_n/R_n$ taken in the category $\mathrm{Sh}_{fppf}({\mathrm{Sch}}/S_n)$ is isomorphic to $Y_n$ by \Cref{algspacepre}. 
The base change \(Z_n \to S_n\) of the projective morphism \(\mathcal{Z} \to \Spec(W(k))\) is projective.
Then the result of Altman-Kleiman recalled in \Cref{AltmanKleiman} above shows that $Y_n$ is a quasi-projective $S_n$-scheme.
Since \(Z_n\) is proper over \(S_n\) and \(Y_n\) admits a surjection from \(Z_n\), we can show that \(Y_n\) is proper over \(S_n\) (\cite[Proposition 12.59]{GW1}) and thus it is projective over \(S_n\).

We want to show that $\{Y_n\}_{n \ge 1}$ admits an algebraization. To this aim, consider the norm $E_n \defeq \Nr_{f_n}(L_n)$ defined from \(f_n \colon Z_n \to Y_n\) (see \citeSta{0BD3} for the existence of norm in our setting and \citeSta{0BCY} for the construction and properties of $\Nr_{f_n}$). Another reference is \cite[Remark 12.25 and Exercise 12.25]{GW1}.
Since $L_n$ is ample and $f_n$ is finite, the proof of \citeSta{0BD0} claims that $E_n$ is ample on $Y_n$. So $\{E_n\}_{n \ge 1}$ forms a $p$-adic formal ample invertible sheaf. Finally, since \(Y_1\) is proper over \(S_1\) and \(E_1\) is an ample line bundle over \(Y_1\), this system provides an algebraization $\mathcal{Y}$ and \(\mathcal{E}\), where \(\mathcal{E}\) is an ample line bundle over \(\mathcal{Y}\) and \(\mathcal{Y}\) is a projective scheme over \(\Spec(W(k))\) by \citeSta{089A} (or \cite[Theorem 8.4.10]{I05}).
Also, the morphism \(\{f_n\}_{n \geq 1}\) of \(p\)-adic formal schemes is uniquely algebraizable (\citeSta{0A42}) since both \(\mcalZ\) and \(\mcalY\) are proper and separated of finite type. So there exists a morphism $\mathcal{Z} \to \mathcal{Y}$ of proper \(W(k)\)-schemes which reduces to the finite \'etale morphism $f \colon X \to Y$ along $\Spec(k) \to \Spec(W(k))$, and $\mathcal{Z} \to \mathcal{Y}$ is a finite \'etale surjection in view of \Cref{FiniteLiftings} and \Cref{OpenNeighborhoodofClosedFiber}.
\end{proof}

\begin{corollary}
\label{LiftingCor}
Let $X$ be a smooth projective variety over a perfect field $k$ of characteristic $p>0$. If $H^2(X,T_{X})=H^2(X,\mathcal{O}_X)=0$, then \(X\) satisfies the condition of the bullet in \Cref{algquotient}(1) and it admits a flat projective lifting \(\mathcal{X}\) over $W(k)$.
In particular, let $X \to Y$ be a surjective finite \'etale morphism such that $[K(X):K(Y)]$ is not divisible by $p$. 
Then there is a finite \'etale surjective morphism $\mathcal{X} \to \mathcal{Y}$ of smooth projective $W(k)$-schemes which is a flat lifting of $X \to Y$ along $\Spec(k) \to \Spec(W(k))$. In particular, $Y$ admits a flat projective lifting over $W(k)$.
\end{corollary}

\begin{proof}
We check the condition of \Cref{algquotient}. Let $X=X_1 \to X_2 \to \cdots \to X_k$ be a sequence such that $X_j$ is a flat $S_j$-scheme and $X_j \cong X_{j+1} \times_{S_{j+1}} S_j$ for any $1 \le j \le k-1$. Then we want to construct $X_k \to X_{k+1}$ such that $X_{k+1}$ is a flat $S_{k+1}$-scheme and $X_{k} \cong X_{k+1} \times_{S_{k+1}} S_k$. 
By \cite[Remark 3.27]{BT24}, the obstruction class \(\ob(X_k)\) lies in \(H^2(X, T_X)\), which vanishes by assumption.
Hence we get $X_{k+1} \to S_{k+1}$ extending $X_k \to S_k$. Since $X$ is projective, then \cite[Corollary 8.5.6]{I05} and the vanishing \(H^2(X, \mathcal{O}_X) = 0\) will give a projective flat scheme $\mathcal{X}$ over $W(k)$ such that $\widehat{\mathcal{X}} \cong \varinjlim_n Z_n$.

By \Cref{algquotient}, there is a $p$-adic formal scheme $Y=Y_1 \to Y_2 \to \cdots$ and there is a $p$-adic formal scheme $X=Z_1 \to Z_2 \to \cdots$ compatible with $Y=Y_1 \to Y_2 \to \cdots$. So it follows from the second assertion of \Cref{algquotient} that $X \to Y$ lifts to $\mathcal{X} \to \mathcal{Y}$, where $\mathcal{Y}$ is the algebraization of $Y=Y_1 \to Y_2 \to \cdots$, which completes the proof of the corollary.
\end{proof}

\begin{remark}
\label{HironakaExamp}
\begin{enumerate}
\item
In connection with the hypothesis of \Cref{algquotient} (2), one might wonder if it is possible to deal with an arbitrary finite \'etale quotient of a proper scheme in the category of schemes. However, this is not necessarily true. In \cite[Example B.3.4.2]{H83}, Hironaka constructed an example of a 3-dimensional non-projective, complete complex variety $X$ which has a fixed-point free involution $\sigma:X \to X$. Then the quotient $X \to X/\langle\sigma\rangle$ is finite \'etale and $X/\langle\sigma\rangle$ is a Moishezon manifold which is not a scheme. In other words, the field of meromorphic functions of $X/\langle\sigma\rangle$ has $\mathbb{C}$-transcendence degree equal to $\dim X$. For more examples, we refer the reader to \cite[Example 14]{K12}. See also \citeSta{0AGG} for the deviation of the presentation of $Y$ as a sheaf $X/R$ beyond the \'etale equivalence relations.

\item 
Even if $X_n \to S_n$ is proved to be proper for $n \ge 0$, the formal scheme $\{X_n\}_{n \ge 0}$ is not necessarily algebraizable, meaning that the formal moduli space lifting $X_0$ over $W(k)$ may be large. Such an example is already known to exist for formal abelian schemes (see \cite[Remarks 8.5.24 (b)]{I05}).
\end{enumerate}
\end{remark}

\section{Smooth projective varieties with quasi-canonical liftings}

Let us start with the next lemma. Although we are mainly interested in the case that $T_{X}$ is trivial, we decided to include the case of numerically flat vector bundles for wide applicability (see \cite{EY23} and \cite{KW23} for these topics).
Recall that a vector bundle $E$ on a smooth variety $X$ is \emph{numerically flat} if $E$ and $E^\vee \defeq {\mathcal{H}om}(E,\mathcal{O}_X)$ are nef.

\begin{lemma}
\label{numericallyflat}
Let $X$ be a smooth projective variety over an algebraically closed field $k$. Then the following assertions hold.
\begin{enumerate}
\item
Assume that the tangent bundle $T_{X}$ is numerically flat (resp. trivial). Assume that $f:Y \to X$ is \'etale. Then the tangent bundle $T_{Y}$ is numerically flat (resp. trivial).

\item Assume that \(X\) is ordinary with trivial cotangent bundle \(\Omega_X^1\), and that \(Y \to X\) is a surjective finite \'etale morphism. Then \(Y\) is ordinary.

\end{enumerate}
\end{lemma}

\begin{proof}
$(1)$: Assume that $T_{X}$ is numerically flat. It is a general fact that the pullback $f^*T_{X}$ is numerically flat on $Y$. Since $f$ is \'etale, we have a short exact sequence $0 \to f^*\Omega^1_X \to \Omega^1_Y \to \Omega^1_{Y/X} \to 0$ and $\Omega^1_{Y/X}=0$ in view of \cite[Proposition 18.18 and Proposition 18.29]{GW2}. So we have an isomorphism $f^*\Omega^1_X \cong \Omega^1_Y$. Since $X$ is a smooth variety, $\Omega^1_X$ is a locally free $\mathcal{O}_X$-module and we get
$$
f^*T_{X} \cong f^*{\mathcal{H}om}(\Omega_X^1,\mathcal{O}_X) \cong {\mathcal{H}om}(f^*\Omega_X^1,\mathcal{O}_Y) \cong {\mathcal{H}om}(\Omega_Y^1, \mathcal{O}_Y) = T_Y
$$
(see, for example, \cite[Exercise 7.10]{GW1}). Hence $T_{Y}$ is numerically flat.

$(2)$: By (1), the cotangent bundle of $Y$ is also trivial. By \Cref{ordinarysplit}, $(X,F_X)$ has a lifting over $W_2(k)$. Hence $(Y,F_Y)$ also has a lifting over $W_2(k)$ by \cite[Lemma 3.3.5]{AWZ21}. Applying \Cref{ordinarysplit} to $Y$, we conclude that $Y$ is ordinary.
\end{proof}

\subsection{The proof of \texorpdfstring{\Cref{mtheorem1}}{Main Theorem 2}}
We will need \Cref{liftcanonical}, \Cref{algquotient}, and the following lemma for the proof of the main result. The following splitting lemma is based on the proof of \cite[Lemma 3.1]{Z17}.

\begin{lemma}
\label{SplittingLemma}
Let \(g \colon Z \to X\) be a surjective finite \'etale morphism of smooth projective varieties over a field \(k\) of characteristic \(p > 0\). If the degree \([K(Z): K(X)]\) of \(g\) is prime to \(p\), then the canonical map \(B\Omega_X^1 \to g_*B\Omega_Z^1\) splits in the category of \(\mathcal{O}_X\)-modules. In particular, we have a split injection
\begin{equation} \label{splitCoh1}
H^k(X, T_X \otimes B\Omega_X^1) \hookrightarrow H^k(Z, T_Z \otimes B\Omega_Z^1)
\end{equation}
for each \(k \in \mathbb{Z}\). Similarly, we obtain the following split injection for a nc pair \((X, D)\) and \((Z, D_Z \coloneqq g^*D)\):
\begin{equation} \label{splitCoh2}
H^k(X, T_X(-\log D) \otimes B\Omega_X^1) \hookrightarrow H^k(Z, T_Z(-\log D_Z) \otimes B\Omega_Z^1)
\end{equation}
for each \(k \in \mathbb{Z}\).
\end{lemma}

\begin{proof}
As in $(\ref{CartierExact2})$, we have short exact sequences:
$$
0 \to \mathcal{O}_Z \to F_*\mathcal{O}_Z \to B\Omega^1_Z \to 0~\mbox{and}~0 \to \mathcal{O}_X \to F_*\mathcal{O}_X \to B\Omega^1_X \to 0.
$$
Since $g$ is an affine morphism, we get the induced exact sequence $0 \to g_*\mathcal{O}_Z \to g_*F_*\mathcal{O}_Z \to g_*B\Omega^1_Z \to 0$ by \citeSta{0G9R}. Since the Frobenius morphism commutes with an arbitrary morphism, we get $g_*F_*\mathcal{O}_Z \cong F_*g_*\mathcal{O}_Z$ by \cite[(7.8.2)]{GW1}, which also gives $g_*B\Omega^1_Z \cong \coker\big(g_*\mathcal{O}_Z \to F_*g_*\mathcal{O}_Z\big)$. Thus, the map of sheaves $\mathcal{O}_X \to g_*\mathcal{O}_Z$ induces a commutative diagram:
\begin{equation}
\label{F-splitdiagram}
\vcenter{
\xymatrix@M=10pt{ 
0 \ar[r] &\mathcal{O}_X \ar[r]^{F^\sharp}\ar[d] &F_*\mathcal{O}_X \ar[r]\ar[d]& B\Omega^1_X \ar[d]\ar[r] &0&\\
0 \ar[r] &g_*\mathcal{O}_Z \ar[r]^{g_*(F^\sharp)} &F_*g_*\mathcal{O}_Z \ar[r]& g_*B\Omega^1_Z \ar[r] &0 \\
}}
\end{equation}
Since \(g \colon Z \to X\) is a finite \'etale surjective morphism between smooth projective varieties of constant degree \(d\coloneqq[K(Z):K(X)]\), which is prime to \(p\), the normalized trace \(\sigma\coloneqq d^{-1}\mathrm{Tr}_g\colon g_*\mathcal{O}_Z\to\mathcal{O}_X\) splits the inclusion \(\mathcal{O}_X\to g_*\mathcal{O}_Z\) by \Cref{tracemap}. \'Etale locally on \(X\), the trace is the sum map, so \(\mathrm{Tr}_g(s^p)=\mathrm{Tr}_g(s)^p\) for every local section \(s\). Since \(d^p=d\), we obtain \(\sigma(s^p)=\sigma(s)^p\). Thus
\[
F^\sharp\circ\sigma=(F_*\sigma)\circ g_*(F^\sharp).
\]
Consequently, \(\sigma\) and \(F_*\sigma\) induce a map \(g_*B\Omega_Z^1\to B\Omega_X^1\) on the cokernels of the horizontal maps in \eqref{F-splitdiagram}. This map splits the canonical map \(B\Omega_X^1\to g_*B\Omega_Z^1\). Hence all vertical maps in \eqref{F-splitdiagram} split.

Taking cohomology, we have an injection:
\begin{equation}
H^k(X,T_{X} \otimes B\Omega^1_X) \hookrightarrow H^k(X,T_{X} \otimes g_*B\Omega^1_Z).
\end{equation}
On the other hand, we can show that \(H^k(X, T_X \otimes g_*B\Omega_Z^1) \cong H^k(Z, T_Z \otimes B\Omega_Z^1)\): we use the fact that \(T_X \otimes g_*B\Omega_Z^1 \cong \mathcal{H}om(\Omega_X^1, g_*B\Omega_Z^1)\) by \cite[Proposition 7.7]{GW1} and $g^*\Omega^1_X \cong \Omega^1_Z$ by the \'etaleness of $g$ as follows:
\begin{align}
H^k(X,T_{X} \otimes g_*B\Omega^1_Z) & \cong H^k(X,\mathcal{H}om(\Omega^1_X, g_*B\Omega^1_Z)) \cong H^k(X,g_*\mathcal{H}om(g^*\Omega^1_X, B\Omega^1_Z)) \nonumber \\
& \cong H^k(Z,\mathcal{H}om(\Omega^1_Z, B\Omega^1_Z)) \cong H^k(Z,T_{Z} \otimes B\Omega^1_Z), 
\label{TangentBundleIsom}
\end{align}
where the second isomorphism follows from \cite[Proposition 7.11]{GW1} by checking locally. The logarithmic variant (\ref{splitCoh2}) is proved in the same way.
\end{proof}

The prime to \(p\) degree condition is stable under composition and base change. The composition is clear, and the base change is given by the following lemma.

\begin{lemma}
\label{StableDegree}
Let \(f \colon X \to Y\) and \(g \colon Z \to Y\) be finite morphisms of integral schemes. Assume that the fiber product \(X \times_Y Z\) is also an integral scheme. If the degree \([K(Z):K(Y)]\) of \(g\) is prime to \(p\), then so is the base change \(X \times_Y Z \to X\).
\end{lemma}

\begin{proof}
Since \(f\) and \(g\) are morphisms of integral schemes, they induce \(\Spec(K(X)) \to \Spec(K(Y)) \leftarrow \Spec(K(Z))\) compatible with \(f\) and \(g\). This makes a unique morphism \(\Spec(K(X) \otimes_{K(Y)} K(Z)) \to X \times_Y Z\) and thus, we can get \(K(X) \hookrightarrow K(X \times_Y Z) \hookrightarrow K(X) \otimes_{K(Y)} K(Z)\). The last term is isomorphic to \(K(X)^{\oplus [K(Z):K(Y)]}\) as a \(K(X)\)-module, so the degree of the extension \(K(X) \hookrightarrow K(X \times_Y Z)\) is a divisor of \([K(Z):K(Y)]\) which is also prime to \(p\).
\end{proof}

\Cref{SplittingLemma} enables us to descend the vanishing from \(H^k(Z, T_Z \otimes B\Omega_Z^1)\) to \(H^k(X, T_X \otimes B\Omega_X^1)\). We are now ready to prove \Cref{mtheorem1}.

\begin{proof}[Proof of \Cref{mtheorem1}]
Our first task is to construct $(\mathcal{X},\widetilde{F}_X)$ and establish its uniqueness. Let $g:Z \to X$ be the good Frobenius lifting cover of \((X, D)\) (\Cref{DefGoodFrobeniusLiftingCover}). We have the vanishing $H^k(Z,T_{Z}(-\log D_Z) \otimes B\Omega^1_Z) \cong 0$ for $k=0,1$ by the definition of good Frobenius lifting cover. Then by \((\ref{splitCoh2})\),
\begin{equation}
\label{ordinaryvanishing}
H^k(X,T_{X}(-\log D) \otimes B\Omega^1_X) \cong 0~\mbox{for}~k=0,1.
\end{equation}

By applying \cite[Variant 3.3.2]{AWZ21} (see also \Cref{MSResultsPropositions} (1)), we get a unique $p$-adic formal nc scheme with Frobenius lift
$$
(X,D,F_X) \to (X_2,D_2,F_{X,2}) \to \cdots \to (X_n,D_n,F_{X,n}) \to \cdots
$$
such that $X_n$ is a flat $S_n$-scheme, $X_n \cong X_{n+1} \times_{S_{n+1}} S_n$ $D_{n+1}|_{X_n}=D_n$ and $F_{X,n+1}|_{X_n}=F_{X,n}$. Since $Z \to X$ is a finite \'etale surjection, \Cref{liftcanonical} may be supplied to yield the following commutative diagram of $p$-adic formal schemes:
\[
\xymatrix@M=10pt@C=7pt{
(Z,D_Z,F_Z)\ar[r] \ar[d]^f & (Z_2,D_{Z_2},F_{Z,2}) \ar[r] \ar[d]^{f_2}&\ar@{.}[r] & \ar[r] & (Z_n,D_{Z_n},F_{Z,n}) \ar[r]\ar[d]^{f_n} & \ar@{.}[r]&\\
(X,D,F_X)\ar[r] & (X_2,D_2,F_{X,2}) \ar[r] & \ar@{.}[r] & \ar[r] & (X_n,D_n,F_{X,n}) \ar[r] & \ar@{.}[r]& \\
}
\]
Now the $p$-adic formal scheme $\{(Z_n,D_{Z_n},F_{Z,n})\}_{n \ge 0}$ is unique among all $p$-adic formal schemes starting with $(Z,D_Z,F_Z)$ in view of \cite[Variant 3.3.2]{AWZ21}. Since the good Frobenius lifting cover \(g \colon Z \to X\) is assumed that \(Z\) admits a quasi-canonical lifting \((\mathcal{Z},\mathcal{D}_Z, \widetilde{F}_Z)\) and \(H^0(Z, T_Z(-\log D_Z) \otimes B\Omega_Z^1) = 0\), the uniqueness of $\{(Z_n,D_{Z_n},F_{Z,n})\}_{n \ge 0}$ yields that $(\mathcal{Z},\mathcal{D}_Z,\widetilde{F}_Z)$ is an algebraization of $\varinjlim_n (Z_n,D_{Z_n},F_{Z,n})$ and this is the canonical lifting of $(Z,D_Z,F_Z)$ by \Cref{CanonicalUnique}.
By the proof of \Cref{algquotient}, one can use the norm of line bundles to conclude that there is a flat proper scheme $\mathcal{X}$ over $W(k)$ such that $\widehat{\mathcal{X}} \cong \varinjlim_{n} X_n$.
By \Cref{liftcanonical}, $\varinjlim_n F_{X,n}$ and \(\varinjlim_n f_n\) can be algebraized to give a unique Frobenius lift $\widetilde{F}_X$ on $\mathcal{X}$ and a finite \'etale surjective morphism \(\widetilde{f} \colon \mathcal{Z} \to \mathcal{X}\) compatible with Frobenius lifts. It remains to show the uniqueness of $(\mathcal{X},\mathcal{D},\widetilde{F}_X)$ up to isomorphism. However, this is shown readily by \((\ref{ordinaryvanishing})\) and \Cref{CanonicalUnique}.
\end{proof}

\begin{corollary} \label{Cormtheorem1}
  Keep the same notation and assumptions as in \Cref{mtheorem1} and assume that \(X\) is normal.
  Denote by $\Pic(\mathcal{X}) \to \Pic(X)$ the map of Picard groups induced by the closed immersion $X \hookrightarrow \mathcal{X}$. Set $P\coloneqq \{\mathcal{L} \in \Pic(\mathcal{X})~|~\widetilde{F}_X^*(\mathcal{L}) \cong \mathcal{L}^p\}$. Then $P$ is a subgroup of $\Pic(\mathcal{X})$ and the composite mapping  $P \hookrightarrow \Pic(\mathcal{X}) \to \Pic(X)$ is an isomorphism.
\end{corollary}

\begin{proof}
  This follows from \Cref{mtheorem1} and \Cref{LineBundleLift}.
\end{proof}

\subsection{The case of ordinary abelian varieties}
In the case that $X$ is an \'etale quotient of an ordinary abelian variety, we have the following result. This is stated in \cite[Remark 3.1.6]{AWZ21} without a proof. Also the existence of a flat lifting (not a quasi-canonical lifting) over \(W(k)\) of \(X\) is shown in \cite[Proposition 4.12]{BBKW22}.

\begin{corollary}
\label{AbVarLiftWitt}
Let $X$ be a smooth projective variety over an algebraically closed field $k$ of characteristic $p>0$. Assume $X$ is a finite \'etale quotient of an ordinary abelian variety. Then $X$ has a quasi-canonical lifting $(\mathcal{X},\widetilde{F}_X)$ over the Witt vectors $W(k)$.

If the degree of the quotient map is prime to $p$, then \((\mathcal{X}, \widetilde{F}_X)\) is the canonical lifting of \((X, F_X)\) over \(W(k)\).
Moreover, we have the functoriality of the canonical liftings:
\begin{itemize}
\item
Let $\psi:X \to Y$ be a morphism of smooth projective varieties over $k$ such that $X$ (resp. $Y$) admits an ordinary abelian variety as a finite \'etale covering whose degrees are prime to \(p\). Then there exists a $W(k)$-morphism $\widetilde{\psi}:\mathcal{X} \to \mathcal{Y}$ such that $\widetilde{\psi}$ is a lifting of $\psi$, $(\mathcal{X},\widetilde{F}_X)$ and $(\mathcal{Y},\widetilde{F}_Y)$ are the canonical liftings of $X$ and $Y$ respectively, and the following diagram commutes
  \[
  \xymatrix@M=10pt{ 
  \mathcal{X} \ar[r]^{\widetilde{F}_X} \ar[d]_{\widetilde{\psi}} &\mathcal{X} \ar[d]_{\widetilde{\psi}}\\
  \mathcal{Y} \ar[r]^{\widetilde{F}_Y} & \mathcal{Y} \\
  }
  \]
\end{itemize}

\end{corollary}

\begin{proof}
By assumption, we can take a finite \'etale surjection $A' \to X$ from an ordinary abelian variety $A'$. Then $A' \to X$ extends to a Galois covering $A \to A' \to X$ (see, for example, \cite[Remark 4.11]{BBKW22}). Hence $A$ is also an ordinary abelian variety by \cite[Theorem at page 168]{Mumford} and \Cref{numericallyflat} (2). So without losing generality, it is sufficient to consider the case when $f:A \to X$ is a $G$-Galois covering from an ordinary abelian variety $A$. This means that $X$ is a universal quotient of $A$ by $G$ in the category of schemes.

We can take the canonical lifting \((\mathcal{A},\widetilde{F}_A)\) of \(A\) (see \cite{MS87} and \cite{Me72}).
By \Cref{MSResultsTheorem}(2), each \(g\in G\) lifts uniquely to an automorphism \(\widetilde{g}\) of \(\mathcal{A}\) commuting with \(\widetilde{F}_A\). Uniqueness gives a \(G\)-action. For \(g\ne1\), the fixed-point locus of \(\widetilde{g}\) is proper over \(W(k)\) and has empty special fiber, hence is empty. Thus the action is free, and \(\mathcal{X}\coloneqq\mathcal{A}/G\) is projective and flat over \(W(k)\), with \(\mathcal{A}\to\mathcal{X}\) finite \'etale; see also \cite[Proposition 4.12]{BBKW22}. The equivariant morphism \(\widetilde{F}_A\) descends to a Frobenius lifting \(\widetilde{F}_X\) on \(\mathcal{X}\). If \(X\) admits a finite \'etale cover by an ordinary abelian variety of degree prime to \(p\), then \Cref{mtheorem1} shows that this lifting is canonical.

We next prove the functoriality. By assumptions, there are finite \'etale covers $A_1' \to X$ and $g:A_2 \to Y$ from ordinary abelian varieties whose degrees are prime to \(p\). Let $X' \defeq X \times_Y A_2$. Then $X' \to X$ is a finite \'etale morphism, because $A_2 \to Y$ is so. Set $A_1 \defeq A_1' \times_X X'$. Since $A_1 \to A_1'$ is the base change of \(A_2 \to Y\) along \(A_1' \to X \to Y\), it is finite \'etale and we see that $A_1$ is an ordinary abelian variety (\Cref{numericallyflat} (2)). Let $f$ denote the composite morphism $A_1 \to A_1' \to X$ and let $\phi$ denote the composite morphism $A_1 \to X'=X \times_Y A_2 \to A_2$.
Note that the prime to \(p\) property of degrees of finite morphisms between integral schemes is stable under composition and base change (\Cref{StableDegree}).
So we have a commutative diagram:
$$
\vcenter{
\xymatrix@M=10pt{ 
A_1 \ar[d]^f \ar[r]^\phi & A_2 \ar[d]^g\\
X \ar[r]^\psi & Y \\
}}
$$
where all maps in the vertical direction are surjective finite \'etale morphisms whose degrees are prime to \(p\).

By the above proof, \((X, F)\) and \((Y, F)\) have quasi-canonical liftings \((\mathcal{X}, \widetilde{F}_X)\) and \((\mathcal{Y}, \widetilde{F}_Y)\), respectively. For each mod \(p^n\)-reduction \((X_n, F_{X,n})\) and \((Y_n, F_{Y,n})\), we have to take a morphism \(f_n \colon X_n \to Y_n\) of \(S_n\)-schemes which is compatible with \(f_{n-1}\), \(F_{X,n}\), and \(F_{Y,n}\) (if it exists, it can be extended to \(\mathcal{X} \to \mathcal{Y}\) uniquely by \citeSta{0A42}). Now according to \Cref{MSResultsPropositions} (3), it suffices to check the vanishing:
$$
H^0(X,\psi^*T_Y \otimes B\Omega^1_X) \cong 0.
$$
Since the degree of \(f \colon A_1 \to X\) is prime to \(p\), we can apply \Cref{SplittingLemma} and we can get an injection
\begin{equation*}
H^0(X,\psi^*T_Y \otimes B\Omega^1_X) \hookrightarrow H^0(X, \psi^*T_Y \otimes f_*B\Omega_{A_1}^1).
\end{equation*}
It suffices to show that the latter group is zero. As in \Cref{SplittingLemma} (\ref{TangentBundleIsom}), we compute
$$
H^0(X,\psi^*T_{Y/k} \otimes f_*B\Omega^1_{A_1}) \cong H^0(X, \mathcal{H}om(\psi^*\Omega^1_Y,  f_*B\Omega^1_{A_1})) \cong H^0(X, f_*\mathcal{H}om(f^*\psi^*\Omega^1_Y,  B\Omega^1_{A_1}))
$$
$$
\cong H^0(X, f_*\mathcal{H}om(\phi^*g^*\Omega^1_Y,  B\Omega^1_{A_1})) \cong H^0(A_1, \mathcal{H}om(\phi^*g^*\Omega^1_Y,  B\Omega^1_{A_1})) \cong 0,
$$
where the first isomorphism follows from \cite[Exercise 7.20]{GW1} and the vanishing in the last step follows from \Cref{MSResultsPropositions}(1) applied to the ordinary abelian variety $A_1$, together with the following fact: Since $g$ is \'etale and the tangent bundle of $A_2$ is trivial, the pullback $\Omega^1_{A_2} \cong g^*\Omega^1_Y$ is trivial, thus $\phi^*g^*\Omega^1_Y$ is also trivial.
%
\end{proof}

\begin{remark}
\begin{enumerate}
\item
If $X$ is an $\mathbb{F}_p$-scheme and $L$ is a line bundle on $X$, then it is true that $F^*_X(L) \cong L^p$. This can be checked by looking at the Frobenius action on $\mathcal{O}_X^\times$ combined with an isomorphism $\Pic(X) \cong H^1(X,\mathcal{O}_X^\times)$. However, it is not always true that $\widetilde{F}_X^*(\mathcal{L}) \cong \mathcal{L}^p$ for $\widetilde{F}_X:\mathcal{X} \to \mathcal{X}$ as constructed in \Cref{mtheorem1}.

\item
There is a non-logarithmic version of \Cref{mtheorem1}, in which case $Z$ is taken to be an ordinary abelian variety. However, the logarithmic case has a wide applicability. For example, one could try to consider the case that $Z$ is a toric fibration. See \cite{AWZ21}, \cite{AWZ23} and \cite{AZ21} for extensive studies of Frobenius liftability question of toric fibrations over $W_2(k)$.

\item 
Let $f:X \to Y$ be a surjective finite \'etale morphism from a smooth projective variety $X$ that is liftable over the Witt vectors $W(k)$. Then is it true that $Y$ is also liftable over $W(k)$? This question has a negative answer. Serre has constructed $X \to Y$, where $X$ is a smooth complete intersection such that the degree of $f$ is divisible by $p$ and $Y$ is not liftable even over $W_2(k)$ (Serre-Godeaux varieties). Such an example also appears in \cite[Remark 3.1.7]{AWZ21}.
By using the flatness criterion using Hilbert polynomial, it can be proved that any global complete intersection projective variety has a flat lifting over the Witt vectors. The main point of the construction is that there is an automorphism of $X$ that does not lift over $W_2(k)$. The details of this construction can be found in \cite[Corollary 8.6.7]{I05}.

\end{enumerate}
\end{remark}

By \cite[Theorem 1 and Lemma 1.1]{MS87}, any smooth projective globally Frobenius-split variety with trivial tangent bundle admits a finite \'etale Galois cover of \(p\)-power degree by an ordinary abelian variety. The following observation concerns the canonical bundle of such quotients by using \Cref{AbVarLiftWitt}.

\begin{corollary} \label{ConverseEtaleQuotient}
Let \(X\) be a smooth projective variety over an algebraically closed field \(k\) of characteristic \(p>0\). If \(X\) admits a finite \'etale surjection of \(p\)-power degree from an ordinary abelian variety, then \(\omega_X\cong\mathcal{O}_X\).
\end{corollary}

\begin{proof}
Let \(f\colon A\to X\) be the given finite \'etale surjection and write \(d=\deg(f)=p^r\). Since \(f\) is \'etale and \(A\) is an abelian variety, we have \(f^*\omega_X\cong\omega_A\cong\mathcal{O}_A\). Since \(f\) is finite locally free of degree \(d\), there is a norm homomorphism \(\Nr_f\colon\Pic(A)\to\Pic(X)\) by \citeSta{0BD2} and \citeSta{0BCY}. Moreover, the latter assertion gives \(\Nr_f(f^*M)\cong M^{\otimes d}\) for every invertible sheaf \(M\) on \(X\). Applying this to \(M=\omega_X\), we obtain
\[
\omega_X^{\otimes d}\cong\Nr_f(f^*\omega_X)\cong\Nr_f(\mathcal{O}_A)\cong\mathcal{O}_X.
\]
Thus \(\omega_X\) is torsion in \(\Pic(X)\), with order dividing \(d\).

By \Cref{AbVarLiftWitt}, \(X\) admits a Frobenius lifting over \(W(k)\), and hence is globally Frobenius-split by \cite[Theorem 5.5]{Z17}. Choose an \(\mathcal{O}_X\)-linear splitting \(\sigma\colon F_*\mathcal{O}_X\to\mathcal{O}_X\) of \(F_X\). In particular, \(\sigma(1)=1\), so \(\sigma\ne0\).

Since \(X\) is smooth over the perfect field \(k\), \(F_X\) is finite and flat. Grothendieck duality for finite morphisms identifies \(\mathcal{H}om_{\mathcal{O}_X}(F_*\mathcal{O}_X,\omega_X)\) with \(F_*\omega_X\); see \citeSta{0AX2} for finite duality and \cite[proof of Lemma (1.1), p.~193]{MS87} for this identification in terms of the Cartier operator. More precisely, the top-degree Cartier operator \(C\colon F_*\omega_X\to\omega_X\) is the duality trace, and the pairing
\[
F_*\mathcal{O}_X\otimes_{\mathcal{O}_X}F_*\omega_X\longrightarrow\omega_X,\qquad a\otimes\eta\longmapsto C(a\eta)
\]
is perfect by \cite[proof of Lemma (1.1), p.~193]{MS87}. Tensoring the resulting duality isomorphism with \(\omega_X^{-1}\) and applying the projection formula, we obtain
\[
\mathcal{H}om_{\mathcal{O}_X}(F_*\mathcal{O}_X,\mathcal{O}_X)
\cong F_*\omega_X\otimes\omega_X^{-1}
\cong F_*(\omega_X\otimes F^*\omega_X^{-1})
\cong F_*\omega_X^{\otimes(1-p)}.
\]
The last isomorphism uses \(F^*M\cong M^{\otimes p}\) for every invertible sheaf \(M\). Consequently, the nonzero homomorphism \(\sigma\) corresponds to a nonzero global section of \(F_*\omega_X^{\otimes(1-p)}\), and hence to a nonzero section \(s\in H^0(X,\omega_X^{\otimes(1-p)})\).

Since \(\omega_X^{\otimes d}\cong\mathcal{O}_X\), we also have \((\omega_X^{\otimes (1-p)})^{\otimes d}\cong\mathcal{O}_X\). Choose such an isomorphism. Under this isomorphism, the section \(s^{\otimes d}\) becomes a nonzero global function on \(X\): it is nonzero at the generic point because \(X\) is integral and \(s\ne0\). Since \(X\) is integral and projective over the algebraically closed field \(k\), we have \(H^0(X,\mathcal{O}_X)=k\). This function is therefore a nonzero constant. Thus \(s^{\otimes d}\), and hence \(s\), is nowhere vanishing. The morphism \(\mathcal{O}_X\to \omega_X^{\otimes (1-p)}\) sending \(1\) to \(s\) is consequently an isomorphism. Taking duals gives
\[
\omega_X^{\otimes(p-1)}\cong\mathcal{O}_X.
\]
Finally, choose integers \(a,b\) with \(ad+b(p-1)=1\), which is possible since \(\gcd(d,p-1)=1\). Using duals for negative tensor powers, we obtain
\[
\omega_X\cong(\omega_X^{\otimes d})^{\otimes a}\otimes(\omega_X^{\otimes(p-1)})^{\otimes b}\cong\mathcal{O}_X.
\]
\end{proof}

What kind of projective variety does occur as a finite \'etale quotient of an ordinary abelian variety? Recently, Ejiri and Yoshikawa in \cite{EY23} (see also \cite{KW23}) proved that a globally Frobenius-split variety with numerically trivial tangent bundle $T_X$, meaning that both $T_X$ and $T_X^\vee$ are nef vector bundles, arises as a finite \'etale quotient of an ordinary abelian variety. Another interesting question is the following.

\begin{question}
Classify all finite \'etale quotients of a smooth projective variety.
\end{question}

There is a topological obstruction for the degree of surjective finite \'etale morphisms. Let $f:X \to Y$ be a surjective finite \'etale morphism and let $d$ be the degree of $f$. Then we have
$\chi(X,\mathcal{O}_X)=d \cdot \chi(Y,\mathcal{O}_Y)$. In particular, if $\chi(X,\mathcal{O}_X)=1$, then there is no non-trivial finite \'etale quotient of $X$.

\begin{example}
It is possible to construct a variety $X$ such that $X$ is an \'etale quotient of some ordinary abelian variety, $T_X$ is not trivial, but numerically flat. We learned this example from S. Yoshikawa through a private communication. Let $E$ be an ordinary elliptic curve over an algebraically closed field of characteristic $3$, and let $f:E \times E \to E \times E$ be an automorphism defined by the matrix
\[
\left(
\begin{array}{rrr}
0 & -1 \\
1 & -1 \\
\end{array}
\right)
\]
Then this has order 3. Choose a 3-torsion point $a \in E$ and consider the translation $T_a:E \to E$. Let $X$ be the quotient of $E \times E \times E$ by the action $T_a \times f$. Then it is shown that $T_X$ is not trivial. It is numerically flat as $X$ is an \'etale quotient of $E \times E \times E$. In particular, $X$ is globally Frobenius-split. The details are found in \cite[Remark 5.6]{ES19} and the references therein.
\end{example}

In \cite{Xin16}, Xin has given a classification of smooth minimal projective surfaces
with liftable Frobenius on $W_2(k)$ with some corrections made in \cite[Theorem 6.9]{AWZ23}. Among them, some (ordinary) hyperelliptic surface has a non-trivial canonical bundle, but its tangent bundle is numerically flat.

\subsection{Lifting of automorphisms}

In \cite{BT24}, Brantner and Taelman have succeeded in applying the method of derived algebraic geometry to prove an existence of flat lifting over the Witt vectors for a large class of smooth Calabi-Yau projective varieties in positive characteristic, including ordinary Calabi-Yau varieties. We prove a result which generalizes \cite[Theorem 4.5]{Sr17} (see also \cite[Proposition 2.6]{LT22}) to the higher dimensional case on the lifting of an automorphism.

\begin{proposition}
Let $X$ be a smooth projective variety defined over a perfect field $k$ of characteristic $p>0$ of dimension $d$ such that $X$ is geometrically integral over $k$. Assume $\omega_X \cong \mathcal{O}_X$ and the following conditions.
\begin{enumerate}
\item
$H^i(X,B\Omega_X^j)=0$ for all $i \ge 0$ and $j=0,1,2$ (Bloch-Kato $2$-ordinarity condition).

\item
$H^d_{et}(X_{\overline{k}},\mathbb{Z}_p)$ is $\mathbb{Z}_p$-torsion free.
\end{enumerate}
Then there is a projective scheme $\mathcal{X}$ flat over $W(k)$, called a canonical lifting of $X$ in the sense of \cite[Definition 8.31]{BT24}. Moreover, if $f:X \to Y$ is a $k$-isomorphism, then it lifts to a $W(k)$-isomorphism $\widetilde{f}:\mathcal{X} \to \mathcal{Y}$, where $\mathcal{X}$ (resp. $\mathcal{Y}$) is the canonical lifting of $X$ (resp. $Y$).
\end{proposition}

\begin{proof}
The existence and construction of $\mathcal{X}$ is \cite[Theorem C]{BT24}. For a given $f:X \cong Y$, we will build a lift over $W(k)$ using the Serre-Tate coordinates constructed in \cite{BT24}. We follow the ideas of the proof of \cite[Proposition 2.6]{LT22}. The deformation functors (more precisely, formal moduli problems) ${\mathrm{Def}}_{Y}$ and ${\mathrm{Def}}_X$ will provide an isomorphism of $p$-adic formal schemes $\widehat{f}:\{X_n\}_{n \ge 1} \cong \{Y_n'\}_{n \ge 1}$ specializing to the isomorphism $f:X = X_1 \cong Y = Y_1$ such that $\mathcal{X}$ is an algebraization of the canonical lifting $\mathcal{X}$.

Since $f_n:X_n \cong Y_n'$ is compatible with $f_{n+k}:X_{n+k} \cong Y_{n+k}'$ for all $k \ge 0$ and there is an ample line bundle $\{L_n\}_{n \ge 1}$ on $\{X_n\}_{n \ge 1}$, we can construct an ample line bundle $\{(f_n)_*L_n\}_{n \ge 1}$ on $\{Y_n'\}_{n \ge 1}$. So by applying \cite[Tag 089A]{Stacks}, there is a projective flat scheme $\mathcal{Y}'$ over $W(k)$ together with an isomorphism $\widetilde{f}:\mathcal{X} \cong \mathcal{Y}'$ such that $\widetilde{f}$ is an algebraization of $\widehat{f}$.

It remains to prove that $\mathcal{Y}'$ is the canonical lifting of $Y$ in the sense of \cite[Definition 8.31]{BT24}. To this aim, let $\{Y_n\}_{n \ge 0}$ be the $p$-adic formal scheme attached to the canonical lifting $\mathcal{Y}$ of $Y$. Let ${\mathrm{ST}}_X(-)$ and ${\mathrm{ST}}_{X,\lambda}(-)$ be Serre-Tate period domains respectively, as introduced in \cite[Definition 8.27]{BT24}, where $\{\lambda:\mathbb{Z}[-1] \to \mathbb{G}_{m,X}\} \in {\mathrm{Mod}}(X_{et},\mathbb{Z})$ corresponds to a choice of an ample line bundle (see \cite[\S~2.6]{BT24} for the notation). By \cite[Definition 8.31]{BT24}, the pair $\{(X_n,L_n)\}_{n \geq 1}$ corresponds to the unit of ${\mathrm{ST}}_{X,\lambda}(W_n(k))$. Since the Serre-Tate period domain is functorial, we have a commutative diagram
$$
\vcenter{
\xymatrix@M=10pt{ 
{\mathrm{ST}}_{Y,f_*(\lambda)}(W_{n+1}(k)) \ar[d]^{f^*} \ar[r] &{\mathrm{ST}}_{Y,f_*(\lambda)}(W_{n}(k))  \ar[d]^{f^*}\\
{\mathrm{ST}}_{X,\lambda}(W_{n+1}(k)) \ar[r] & {\mathrm{ST}}_{X,\lambda}(W_{n}(k)) \\
}}
$$
where the horizontal map is induced by the truncation $W_{n+1}(k) \to W_n(k)$ and the unit element of ${\mathrm{ST}}_{Y,f_*(\lambda)}(W_{n}(k))$ is $(Y_n',(f_n)_*L_n)$. Since $f^*$ is an equivalence, which sends the unit to the unit, it follows that $Y_n=Y_n'$. So we conclude that $\mathcal{Y}'$ is the canonical lifting of $Y$, as desired.
\end{proof}



\begin{equation*}
\end{equation*}

\end{document}